\documentclass{amsart}
\usepackage[utf8]{inputenc}
\usepackage{amsfonts,bbm, soul}
\usepackage{hyperref}
\hypersetup{
pdftitle={(SP)-inclusions of C*-algebras},
pdfsubject={Mathematics, Operator Algebras, C*-algebras},
pdfauthor={Gow, Alec},
pdfkeywords={C*-algebras, inclusions of C*-algebras, 46L05}
}
\usepackage{amsmath}
\usepackage{xcolor}
\usepackage{amsthm}
\usepackage{pdflscape}
\usepackage{leftindex}
\usepackage{mathrsfs}
\usepackage{mathtools}
\usepackage{enumerate}
\usepackage{xcolor}
\usepackage{tikz, tikz-cd}
\usepackage{standalone}

\usepackage{latexsym,amsfonts}
\usepackage{amsbsy}
\usepackage{amssymb,amsmath}
\usepackage{amscd}
\usepackage{mathrsfs}
\usepackage{physics}

\newtheorem{mainthm}{Theorem}

\newtheorem{thm}{Theorem}[section]
\newtheorem{cor}[thm]{Corollary}
\newtheorem{prop}[thm]{Proposition}
\newtheorem{lem}[thm]{Lemma}

\newtheorem{quest}[thm]{Question}

\theoremstyle{definition}
\newtheorem{maindefn}[mainthm]{Definition}
\newtheorem{defn}[thm]{Definition}

\newtheorem{exmp}[thm]{Example}

\newtheorem{notn}[thm]{Notation}

\theoremstyle{remark}
\newtheorem{rem}[thm]{Remark}
\newtheorem{rems}[thm]{Remarks}

\newcommand{\F}{\mathbb{F} }
\newcommand{\Z}{\mathbb{Z} }
\newcommand{\G}{\mathbb{G} }

\DeclareMathOperator{\Proj}{Proj}

\makeatletter
\let\c@equation\c@thm
\makeatother
\numberwithin{equation}{section}

\allowdisplaybreaks

\begin{document}
\title[Property (SP) and inclusions of C*-algebras]{Property (SP) and inclusions of C*-algebras}

\author{Alec Gow}
\address{Department of Pure Mathematics and Institute for Quantum Computing, University of Waterloo, 200 University Ave W, Waterloo, Ontario, Canada, N2L 3G1}
\email{a2gow@uwaterloo.ca}

\author{Roberto Hern{\'a}ndez Palomares}
\address{Department of Mathematics, The Ohio State University, 100 Math Tower 231 West 18th Avenue Columbus, OH, United States of America, 43210-1174}
\email{hernandezpalomares.1@osu.edu}

\begin{abstract}
 As a generalization of Gabe and Neagu's inclusions of real rank zero, we introduce a notion of Property (SP) for inclusions of C*-algebras. We begin the systematic study of \emph{SP-inclusions}, establishing a topological description in the commutative setting and permanence under many constructions. 
 We also extend known results about the permanence of Property (SP) for C*-algebras under classical symmetries to the setting of quantum symmetries; we show that Property (SP) for (simple, separable) C*-algebras is preserved under irreducible C*-discrete inclusions.
 Finally, we resolve (in the negative) a question raised by Gabe and Neagu about inclusions coming from Furstenburg boundaries. 
 Namely, we show the existence of a countable, discrete, non-amenable group $G$ such that $C_r^*(G) \subseteq C(\partial_F G) \rtimes_r G$ is not an SP-inclusion (and therefore not an inclusion of real rank zero).
\end{abstract}
\maketitle
\vspace{-1em}
\section{Introduction}\label{sec:Intro}
\renewcommand*{\themainthm}{\Alph{mainthm}}

Inclusions of operator algebras have been major objects of interest for nearly 50 years. 
For von Neumann algebras, Jones' \cite{MR696688} study of inclusions of type-$\rm{II}_1$ factors gave rise to the modern theory of subfactors, which has found far reaching applications across various areas of mathematics and physics (see, e.g., \cite{MR1134131}, \cite{MR4642115}, and references therein).
Inclusions of C*-algebras have also been studied from various standpoints that often benefit from the advances on the von Neumann side. 
For instance, Pimsner and Popa \cite{MR860811} generalized aspects of Jones' \emph{index}--a relative dimension valued in $\{4\cos^2\pi/n\}_{n\geq 3}\cup[4,\infty)$--suitable for an analytic study of C*-inclusions and Watatani \cite{MR996807} developed an algebraic approach compatible with a module theoretic toolkit.

The study of C*-inclusions has experienced rapid growth since, producing deep interactions between index theory and the internal structure of C*-algebras \cite{MR1742862, MR1228532, MR1604162, MR1642530, MR1900138, MR2085108, MR3145747, MR4599249, MR4717816, MR5069553}, along with a Galois theory of C*-dynamical systems \cite{MR4010423, MR4368671, MR4813137, HPN2}. 
These developments revealed that inclusions of operator algebras \emph{have quantum symmetry} and this is best attested to by \emph{discrete} inclusions.

\emph{Discrete subfactors} were introduced by Izumi, Longo and Popa \cite{MR1622812} as a generalization of the subfactors arising from crossed products by discrete groups and those with finite Jones index. 
Later, C. Jones and Penneys \cite{MR3948170} developed a \emph{unitary tensor category} (UTC) description, showing discrete subfactors form the largest class of subfactors determined by a \emph{standard invariant} axiomatized as an action of a UTC together with certain algebraic data--similar to the Tannaka-Krein reconstruction for compact quantum groups \cite{MR943923, MR3204665}.

The class of \emph{C*-discrete inclusions}, introduced by the second-name author and Nelson in \cite{MR5051789}, is a vast generalization of discrete subfactors, forming the largest class of inclusions of C*-algebras determined by a standard invariant.
These inclusions arise canonically as a reduced crossed product of a C*-algebra acted on by a UTC; consequently, inclusions of finite Watatani index, crossed products by discrete (quantum) groups, and certain endomorphic crossed products are all C*-discrete \cite{MR2032486, MR1966826, HHP25, MR5042209, HPN2}.
Therefore, C*-discrete inclusions provide a broad framework to study properties of C*-algebras that persist under dynamical processes dominated by a (quantum) symmetry. 

Another major theme has been the study of an operator algebra through the (partially ordered) set of its projections. 
This perspective dates back to the foundational work of Murray and von Neumann, in which projections were used to develop the noncommutative theory of dimension and to classify factors \cite{MR1503275}, \cite{MR1501899}. In the 1970's, Elliott used the rich structure of projections in AF-algebras to build a complete invariant \cite{MR397420}, thereby initiating the modern K-theoretic approach to classification of C*-algebras (see e.g., the survey \cite{2023arXiv231114238G} and references therein).  
As such, the families of C*-algebras containing ``many'' projections are of great interest.

One of the weakest regularity conditions that can be placed on a C*-algebra to guarantee the existence of an abundance of nonzero projections is \emph{Property (SP)}, which was introduced by name in \cite{BlackadarNotes}.  
A C*-algebra is said to have Property (SP) if each of its nonzero hereditary subalgebras contains a nonzero projection. 
It follows that well-known classes of C*-algebras, like von Neumann algebras, AW$^*$-algebras, and AF-algebras all have Property (SP) \cite{MR1292013}.

From the point of view of classification theory, Property (SP) is too weak to be useful--at least when considered in the absence of additional properties. %
As such, the study of C*-algebras with Property (SP) has been somewhat neglected when compared to the study of their close cousins, the C*-algebras with \emph{real rank zero}. The latter has been characterized in multiple ways \cite{MR1120918}; one such characterization is particularly nice for our purposes, and so we take it to be our definition. Namely, we say that a C*-algebra has real rank zero if each nonzero hereditary subalgebra contains an approximate identity consisting of projections. 
It is clear from the definition that C*-algebras with real rank zero have Property (SP). 

In this paper, we study inclusions of C*-algebras and Property (SP) from two perspectives.
\begin{enumerate}
    \item We develop a relative notion of Property (SP) for inclusions of C*-algebras.
    \item We establish the permanence of Property (SP) for C*-algebras under inclusions.
\end{enumerate}

For the first perspective, we draw inspiration from Gabe and Neagu \cite{GN25} who recently introduced a notion of real rank zero for inclusions of C*-algebras; we introduce a corresponding notion of Property (SP) for inclusions of C*-algebras.

\begin{maindefn}[Definition \ref{defn:relativeSP}] \label{maindefn:relativeSP}
     We say that an inclusion of C*-algebras $A \subseteq B$ is an \emph{SP-inclusion} if for any nonzero element $a \in A_+$ the hereditary C*-subalgebra $\overline{aBa}$ of $B$ contains a nonzero projection.
\end{maindefn}

Our first aim is to justify that this is indeed the correct notion of Property (SP) for inclusions. 
The following result can reasonably be interpreted as such.
\begin{mainthm}[Theorem \ref{thm:mainSP}]\label{mainthm:mainSP}
Let $A\subseteq B$ be an inclusion of C*-algebras. If there exists an intermediate C*-algebra with Property (SP), then $A\subseteq B$ is an SP-inclusion. In particular, if either $A$ or $B$ has Property (SP), then  $A\subseteq B$ is an SP-inclusion. Moreover, a C*-algebra $C$ has Property (SP) if and only if $C\subseteq C$ is an SP-inclusion.
\end{mainthm}

We note that an inclusion of C*-algebras may be an SP-inclusion in the case where both, one, or neither of the algebras in the inclusion have Property (SP) as illustrated by Examples \ref{exmp:both}, \ref{exmp:irrationalRotation} and \ref{exmp:commutative}, respectively. 
Furthermore, we characterize SP-inclusions in topological terms when both of the algebras involved are commutative in Proposition \ref{prop:TopologicalSPInclusion} (cf. \cite[Theorem C]{GN25}). 

Through the systematic development of the theory of SP-inclusions, which are a strict generalization of inclusions of real rank zero (see Proposition \ref{prop:RR0}), we are able to give a negative answer to a question raised in \cite[Remark 5.10]{GN25}.  

\begin{mainthm}[Theorem \ref{thm:Furstenberg}]\label{mainthm:Furstenberg}
     Let $N$ be a discrete group with trivial amenable radical and let $H \neq \{e\}$ be a  discrete torsion-free abelian group. If $G:= N \times H$, then the canonical embedding $C_r^*(G) \subseteq C(\partial_FG) \rtimes_r G$ is not an SP-inclusion. In particular, if $G$ is a countable, discrete, non-amenable group, then the embedding $C_r^*(G) \subseteq C(\partial_F G) \rtimes_r G$ need not be an SP-inclusion (and therefore not an inclusion of real rank zero). 
\end{mainthm}

Working from the second perspective, we draw on classical results about the permanence of Property (SP) for C*-algebras, as studied in e.g., \cite{MR1623278} and \cite{MR1862184}.
Our first result along these lines is an extension of the work of Osaka \cite{MR1862184} that upgrades permanence of Property (SP) for C*-algebras (under certain ``nice'' faithful conditional expectations) to the setting of SP-inclusions.

\begin{mainthm}[Theorem \ref{thm:pinchingMain}]\label{mainthm:pinchingMain}
     Let $A \subseteq B$ be an inclusion of C*-algebras with faithful conditional expectation $E:B\twoheadrightarrow A$. 
     Suppose that $E$ has the pinching property. 
     Then $A \subseteq B$ is an SP-inclusion if and only if $B$ has Property (SP). 
\end{mainthm}

We are also able to show permanence of Property (SP) under a much broader class of inclusions, namely the C*-discrete inclusions introduced in \cite{MR5051789}. 

\begin{mainthm}[Theorem \ref{thm:discreteSP}]\label{mainthm:discreteSP}
     Let $A \subseteq B$ be an irreducible C*-discrete inclusion (by definition, this is equipped with a faithful conditional expectation $E: B \twoheadrightarrow A$). If $A$ is simple, separable, and has Property (SP), then $B$ has Property (SP).
\end{mainthm}

The paper is organized as follows. 
In Section \ref{sec:basic} we introduce the notion of SP-inclusions and establish some basic results, including a topological description of our inclusions in the commutative setting. 
We continue by showing several permanence properties of these inclusions in Section \ref{sec:permanenceProperties}. 
In Section \ref{sec:Symmetries}, we study the permanence of Property (SP) for C*-algebras under quantum symmetries. 
First, we extend classical results to the (relative) SP-inclusion setting and then we show that, under mild assumptions, Property (SP) is permanent under C*-discrete inclusions. 
We conclude with the study of inclusions coming from the Furstenberg boundary in Section \ref{sec:Boundaries}.
In particular, we study the embedding $C_r^*(G) \subseteq C(\partial_FG) \rtimes_r G$ and describe sufficient conditions on  $G$ for the embedding to be an SP-inclusion. \\

\noindent {\bf Acknowledgments.} The authors were supported by an NSERC Discovery Grant held by Michael Brannan. 
AG would like to thank Matthew Kennedy for his helpful comments on the construction of Example \ref{exmp:furstenberg}. 

\tableofcontents

\section{Basic Properties}\label{sec:basic}

Recall that a C*-algebra $A$ is said to have {\bf Property (SP)} if every hereditary C*-subalgebra of $A$ contains a nonzero projection. This is a strictly weaker property than having real rank zero (RR0); $A$ has {\bf real rank zero} if every hereditary C*-subalgebra of $A$ contains an approximate identity of projections. 

\begin{notn}
    Let $A$ be a C*-algebra, $\lambda \geq 0$, and $\varepsilon > 0$. Let us define the function $((\cdot) - \varepsilon)_+:[0,\infty)\to [0,\infty)$
    \[ (\lambda-\varepsilon)_+ = \begin{cases}
        0 \quad & 0 \leq \lambda \leq \varepsilon \\
        \lambda - \varepsilon & \lambda > \varepsilon
    \end{cases}\]
    Using functional calculus, we can define $(a-\varepsilon)_+ \in A_+$ for any $a \in A_+$. 
\end{notn}

The following characterization of C*-algebras with Property (SP) is well-known (e.g., it is stated in \cite[Definition 6.1.1]{MR1292013}). However, as we were unable to locate a proof in the literature, and the proof illustrates the sort of spectral calculus arguments which will be useful in the sequel, we include one here.

\begin{lem}\label{lem:equivalentSP}
    Given a C*-algebra $A$, the following are equivalent:
    \begin{enumerate}[(i)]
        \item $A$ has Property (SP)
        \item For any nonzero element $a \in A_+$, there exists $0 \neq p \in \Proj(A)$ and $\lambda > 0$ such that $a \geq \lambda p$.
    \end{enumerate}
\end{lem}

\begin{proof}
    First, let us show $(ii) \Rightarrow (i)$. That is, we assume that for all $0 \neq a \in A_+$, there exist $0 \neq p \in \Proj{(A)}$, $\lambda > 0$ such that $a \geq \lambda p$.
    Let $B \subseteq A$ be a hereditary subalgebra and let $0 \neq b \in B_+$. 
    Then there exists some $\lambda > 0$, $0 \neq p \in \Proj{(A)}$ such that $b \geq \lambda p$ and $$0 \leq p \leq \lambda^{-1}b \in B.$$
    Since $B$ is a hereditary subalgebra of $A$, this means that $p \in B$.
    So any hereditary subalgebra of $A$ contains a nonzero projection, i.e., $A$ has Property (SP), as required.
    
    To show $(i) \Rightarrow (ii)$, we assume that $A$ has Property (SP).
    Take $0\neq a \in A_+$ and pick $\varepsilon \in (0, \norm{a})$. Then $0 \neq b := (a-\varepsilon)_+ \in A_+$. 
    By the bijective correspondence between hereditary subalgebras of $A$ and open projections in $A^{**}$ (see e.g., \cite[Paragraph 3.11.10]{MR3839621}), there is a unique open projection $q \in A^{**}$ such that $$\overline{bAb} = qA^{**}q \cap A.$$ 

    By assumption, it follows that $qA^{**}q \cap A$ contains some nonzero projection $p.$
    In particular, we have $0 \neq p \leq q$ in $A^{**}$.
    
    We need to show that $\varepsilon p \leq a$ in $A$.
    First, we will show that the support projection of $b$ in $A^{**}$, namely $s(b):=\inf\left\{e \in \Proj{(A^{**})}: eb = b = be \right\}$, is exactly $q$.
    It is clear from the definition that $s(b) \leq q$.
    To show the other inequality, first note that every element of $\overline{bAb}$ is supported by $s(b)$: for all $x \in A$ it is clear that $s(b)(bxb)s(b) = bxb.$
    Now, let $(e_\lambda)_{\lambda \in \Lambda}$ be a contractive approximate identity in $\overline{bAb}$, so that $(e_\lambda) \to q$ in the weak$^{*}$-topology of $A^{**}$.
    We have $e_\lambda = s(b)e_\lambda s(b)$ for each $\lambda \in \Lambda$ and left (resp. right) multiplication by a fixed element is weak${^*}$-continuous, passing to the weak$^{*}$-limit gives $q = s(b)qs(b)$.
    Since $s(b)$ and $q$ are both projections, this means $q \leq s(b)$.
    In particular, we can conclude that $q = s(b).$
    That is, 
    \[q = s((a-\varepsilon)_{+})  = \mathbf{1}_{(\varepsilon, \infty)}(a).\]
    Therefore, we have $\varepsilon q \leq a$ in $A^{**}$.
    Putting our inequalities together, it follows that $\varepsilon p \leq \varepsilon q \leq a$ in $A^{**}$.
    In particular, since $a, p \in A$ and the order in $A$ agrees with the order in $A^{**}$, we get $\varepsilon p \leq a$ in $A$, as required. 
\end{proof}

\begin{rem}
    While we defined Property (SP) for a C*-algebra to be the property that \emph{every} nonzero hereditary C*-subalgebra contains a nonzero projection, it is sufficient to only check hereditary subalgebras generated by nonzero positive elements. (If $A$ is separable, then every hereditary C*-subalgebra is generated by some nonzero positive element, so this is immediate).

    Let us suppose that $A$ is a C*-algebra in which $\overline{aAa}$ contains a nonzero projection for any $0 \neq a \in A_+.$ 
    We wish to show that $A$ has Property (SP). 
    
    Let $H$ be an arbitrary nonzero hereditary subalgebra of $A$.
    Pick some $0 \neq x \in H$ and define \[0 \neq b:= xx^* \in H_+.\]
    By our assumption, along with the fact that $\overline{bAb}$ is the smallest hereditary subalgebra of $A$ containing $b$ (see, e.g., \cite[Proposition II.3.4.2 (ii)]{Blackadar}), we see that there is a nonzero projection
    \[p \in \overline{bAb} \subseteq H,\] as required. 

    From now on, we effectively take the definition of Property (SP) to be those C*-algebras in which any hereditary subalgebra generated by a nonzero positive element contains a nonzero projection.
\end{rem}

\begin{defn}[Definition \ref{maindefn:relativeSP}]\label{defn:relativeSP}
    We say that an inclusion of C*-algebras $A \subseteq B$ is an \emph{SP-inclusion} if for any nonzero element $a \in A_+$ the hereditary C*-subalgebra $\overline{aBa}$ of $B$ contains a nonzero projection. 
\end{defn}

The characterization of Property (SP) for C*-algebras given in Lemma \ref{lem:equivalentSP} has the following analogue for SP-inclusions. 
We omit the proof, as it is essentially the same as above.

\begin{prop}\label{prop:equivalentInclusion}
    For an inclusion of C*-algebras $A \subseteq B$, the following are equivalent:
    \begin{enumerate}[(i)]
        \item $A \subseteq B$ is an SP-inclusion
        \item For any nonzero element $a \in A_+$, there exists $0 \neq p \in \Proj(B)$ and $\lambda > 0$ such that $a \geq \lambda p$ in $B$.
    \end{enumerate} 
\end{prop}

With this characterization in hand, we are able to state and prove our first main result, which justifies that Definition \ref{defn:relativeSP} is the correct analogue of Property (SP) for inclusions.

\begin{thm}[Theorem \ref{mainthm:mainSP}]\label{thm:mainSP}
Let $A\subseteq B$ be an inclusion of C*-algebras. If there exists an intermediate C*-algebra with Property (SP), then $A\subseteq B$ is an SP-inclusion. In particular, if either $A$ or $B$ has Property (SP), then  $A\subseteq B$ is an SP-inclusion. Moreover, a C*-algebra $C$ has Property (SP) if and only if $C\subseteq C$ is an SP-inclusion.
\end{thm}

\begin{proof}
    Let $A\subseteq B$ be an inclusion of C*-algebras and let $C$ be a C*-algebra with Property (SP) such that $A \subseteq C \subseteq B$.
    Then for any $0 \neq a \in A_+$, by definition of Property (SP), $\overline{aCa}$ contains a nonzero projection. 
    Since $\overline{aCa} \subseteq \overline{aBa}$, the latter must also contain a nonzero projection and so $A \subseteq B$ is an SP-inclusion.
    So, if either $A$ or $B$ have Property (SP), then so does the inclusion $A \subseteq B$. 
    
    It follows from above that if $C$ has Property (SP) then the inclusion $C \subseteq C$ does too.
    Conversely, suppose now that $C\subseteq C$ is an SP-inclusion. That $C$ must have Property (SP) is immediate from Lemma \ref{lem:equivalentSP} and Proposition \ref{prop:equivalentInclusion}.
\end{proof}

\begin{exmp}\label{exmp:both}
From Theorem \ref{thm:mainSP}, we obtain a myriad of examples of SP-inclusions by considering any subalgebra $A$ of a C*-algebra $B$ with Property (SP). For example, if $\pi: A \to B(H)$ is a non-degenerate $\ast$-representation of an abstract C*-algebra $A$, then $\pi(A) \subseteq B(H)$ is an SP-inclusion. If $A$ has Property (SP) and $\pi$ is faithful, then both algebras in the inclusion $\pi(A) \subseteq B(H)$ have Property (SP).
\end{exmp}

\begin{exmp}\label{exmp:irrationalRotation}
    By \cite[Remark 6]{MR1247990}, the irrational rotation algebra $A_\theta$ has Property (SP), and by Theorem \ref{thm:mainSP}, the inclusion
    $$
    C(\mathbb{T}) \subseteq A_\theta
    $$
    is an SP-inclusion, despite $C(\mathbb{T})$ not having Property (SP).
    Recall also that $A_\theta \cong C(\mathbb{T})\rtimes_\theta \mathbb{Z}$ (see e.g., \cite[Example VIII.1.1]{MR1402012}); we will have more to say about Property (SP) and inclusions arising as crossed products in Sections \ref{sec:Symmetries} and \ref{sec:Boundaries}.
\end{exmp}

We should justify that SP-inclusions are a generalization of the inclusions of real rank zero introduced by Gabe and Neagu in \cite{GN25}. 
Recall that they defined $A \subseteq B$ to be an \emph{inclusion of real rank zero} if for any nonzero-element $a \in A_+$ the hereditary C*-subalgebra $\overline{aBa}$ contains an approximate identity of projections. 
Keeping in  line with our terminology, we will call such an inclusion a {\bf RR0-inclusion}.

\begin{prop}\label{prop:RR0}
    If $A \subseteq B$ is a  RR0-inclusion, then it is an SP-inclusion. However, the converse need not hold.
\end{prop}

\begin{proof}
    Let $0 \neq a \in A_+$. If $A \subseteq B$ is a RR0-inclusion, then by definition $\overline{aBa}$ contains an approximate unit consisting of projections. In particular, it certainly contains a nonzero projection. So $A \subseteq B$ is an SP-inclusion.

    Let $C$ be a C*-algebra which has Property (SP) but does not have real rank zero (e.g., certain Goodearl algebras, as described in Example \ref{exmp:Goodearl}, and R{\o}rdam's examples of simple C*-algebras with both finite and infinite projections \cite{MR2238291}). 
    Then $C \subseteq C$ is an SP-inclusion by Theorem \ref{thm:mainSP}. However, by \cite[Lemma 1.9]{GN25}, $C \subseteq C$ is an RR0-inclusion if and only if $C$ has real rank zero.
\end{proof}

Proposition \ref{prop:RR0} tells us that we can find examples of SP-inclusions by considering any of the RR0-inclusions which are constructed in \cite{GN25}. 
 In particular, we can find SP-inclusions $A \subseteq B$ where neither $A$ nor $B$ have Property (SP).
With a view towards constructing such an example, we will take a brief detour to consider the commutative setting.

\subsection{Commutative C*-algebras and Inclusions} Let us first consider the topological properties of $X$ for the commutative C*-algebras $C_0(X)$ which are of interest to us.

It is known that for any locally compact Hausdorff space $X$, $C_0(X)$ has Property (SP) if and only if every open subset of $X$ contains a compact open subset \cite{MR1292013}.
On the other hand, recall that $C_0(X)$ has real rank zero if and only if $X$ is \emph{totally disconnected}; that is, $X$ has a basis of compact open sets. 
Both of these characterizations follow from the correspondence between open subsets of $X$ and hereditary subalgebras of $C_0(X)$, as well as the well-known fact that projections in $C_0(X)$ are characteristic functions of compact open sets. 

\begin{rems} Some remarks about these topological properties are in order.
    \begin{enumerate}[(i)]
        \item In the topological literature, the property ``each open subset of $X$ contains a compact open subset'' is sometimes referred to by saying that $X$ has a \emph{$\pi$-basis} of compact open sets (e.g., \cite[p.227]{MR2049453}). We note that there does not appear to be a standard term for a topological space with a $\pi$-basis of compact open sets.
        \item For locally compact Hausdorff spaces, $X$ having \emph{topological dimension zero} is equivalent to being totally disconnected. 
        This justifies Brown and Pedersen's motivating perspective that real rank zero is a non-commutative analogue of topological dimension zero \cite{MR1120918}.  
    \end{enumerate} 
\end{rems}

Somewhat surprisingly, an explicit example of {\bf a commutative C*-algebra with Property (SP) but \emph{not} having real rank zero} seems to have been overlooked in the operator algebraic literature. 
Such examples may be well-known to experts; nevertheless, we describe some below for posterity.

\begin{exmp} 
    Let $X = C_1 \sqcup C_2$ be the set consisting of two concentric circles $$C_i := \{ (i, \theta): 0 \leq \theta < 2\pi\},$$ for $i = 1, 2,$
    let $p: C_1 \to C_2$ be the bijection defined by $$(1, \theta) \mapsto (2, \theta),$$
    and let $U(x,n) \subseteq C_1$ be the open arc of length $\frac{1}{n}$ ($n \in \mathbb{N}$) which is  centered at $x \in C_1$.
    We equip $X$ with a certain topology so that each $x \in C_2$ is isolated and the open basic neighbourhoods of $x \in C_1$ are given by $$B(x, n) = U(x,n) \cup p\left[U(x,n) \setminus \{x\}\right], \, n \in \mathbb{N};$$ see Figure \ref{fig:nbhood} for the latter.
     This space is called \emph{Alexandroff's double circle} or sometimes simply \emph{concentric circles} (see \cite[Example 97]{MR1382863} for details). 
    
  \begin{figure}[!ht]
    \centering
    \adjustbox{scale=0.85}{\begin{tikzpicture}\tikzset{
    white dot/.style={draw=black, fill=white, circle, inner sep=0pt, minimum size=3.5pt},
    black dot/.style={draw=black, fill=black, circle, inner sep=0pt, minimum size=3.5pt}
  }
    \draw[thick] (0,0) circle (1);
    \draw[thick] (0,0) circle (2);

    \draw [ultra thick] (0,0) ++(-15:1) arc[start angle=-15,delta angle=50,radius=1];

    \draw [ultra thick] (0,0) ++(-15:2) arc[start angle=-15,delta angle=50,radius=2];

    \coordinate (0) at (0,0);

    \draw [densely dotted, thick] (0) -- (35:2);
    \draw [densely dotted, thick] (0) -- (-15:2);

    \node[black dot, label = right: $x$] at (10:1) {}; 
    \node[white dot, label = right: $p(x)$] at (10:2) {}; 

    \node[white dot] at (-15:1) {};
    \node[white dot] at (-15:2) {};

    \node[white dot] at (35:1) {};
    \node[white dot] at (35:2) {};

    \node[] at (220:1.3) {$C_1$};
    \node[] at (220:2.3) {$C_2$};

\end{tikzpicture}}
       \caption{$B(x,n)$ for $x \in C_1$, $n \in \mathbb{N}$}
    \label{fig:nbhood}
    \end{figure}
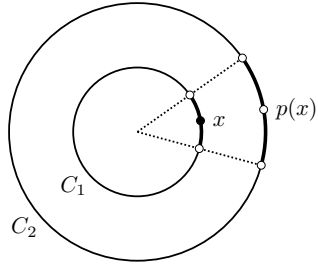

    It is known that $X$ is compact, Hausdorff, and not totally disconnected, so $C(X)$ is a commutative, unital C*-algebra, but it cannot have real rank zero.
    However, we will show that every open subset of $X$ contains a clopen subset, i.e., $C(X)$ has Property (SP).

    Let $x \in X$. If $x \in C_2$ and $N_x$ is an open neighbourhood of $x$, then $\{x\} \subseteq N_x$. 
    Since the points of $C_2$ are isolated and $X$ is Hausdorff, $\{x\}$ is clopen, so we are done.
    If, instead, $x \in C_1$ and $N_x$ is an open neighbourhood of $x$, then there is some $n \in \mathbb{N}$ such that $B(x,n) \subseteq N_x$.
    Moreover, $B(x,n)$ contains points of $C_2$ via the mapping of the arc $U(x,n) \subseteq C_1$ to $C_2$ under the bijection $p$. 
    In particular, there is $y \in C_2$ such that $\{y\} \subseteq B(x,n) \subseteq N_x.$
    As we argued above, $\{y\}$ is a clopen set, so any open neighbourhood of $x$ contains a clopen subset.
    Since any open subset $U \subseteq X$ is an open neighbourhood of each $x \in U$, we are done.
\end{exmp}

\begin{rem}\label{rem:isolatedPoints}
    Somewhat implicit in the preceding example is the following observation: if $X$ is a compact Hausdorff space containing a dense subset of isolated points, then every nonempty open subset of $X$ contains a clopen singleton. 
    In particular, these assumptions are sufficient for showing that $X$ has a $\pi$-basis of clopen sets.
\end{rem}

In light of this example--since Alexandroff's double circle is non-metrizable--we consider here a question posed to us by Mikael R{\o}rdam: if $X$ is a compact, Hausdorff, (completely) metrizable space such that $C(X)$ has Property (SP), does it necessarily have real rank zero?  
Under this set of assumptions on $X$\footnote{Metrizability of $X$ is required for this equivalence if we adopt the convention that an AF-algebra must be separable. If we allow for the existence of non-separable AF-algebras, then one can show more generally that $C(X)$ has real rank zero if and only if $C(X)$ is an AF-algebra (see e.g., \cite[Proposition 2.18]{MR3987301}).}, it is known that $C(X)$ has real rank zero if and only if $C(X)$ is an AF-algebra (see e.g., \cite[Paragraph 6.1.3]{MR1292013}). 

\smallskip
The following example shows that Property (SP) remains weaker than real rank zero for commutative C*-algebras $C(X)$, even under the additional assumption that $X$ is metrizable. 

\begin{exmp} We proceed with a construction  inspired by the \emph{metrizable compactification of the natural numbers} as described in the proof of \cite[Proposition 2.1]{MR2213623}.

  Let $\{q_k\}_{k=1}^{\infty}$ be an enumeration of $\mathbb{Q} \cap [0,1].$
  For each $n \in \mathbb{N}$, define \[A_n := \{q_1, \ldots, q_n\} \times \left\{\frac{1}{n}\right\}.\] 
  Then take \[A := \bigcup_{n=1}^{\infty} A_n,\] which is a countable, discrete subspace of $[0,1] \times [0,1].$
  
  Now define a (closed, bounded) subspace of $[0,1] \times [0,1]$ by
  \[X := [0,1] \times \{0\} \cup A.\] 
  It follows immediately that $X$ is compact, Hausdorff, and metrizable.
  Moreover, it is clear from our construction that $A$ is dense in $X$ and every point of $A$ is isolated in $X$. 

  By Remark \ref{rem:isolatedPoints}, every nonempty open subset of $X$ contains a clopen subset; in particular, $C(X)$ has Property (SP).
  That $C(X)$ does not have real rank zero follows from the observation that \[1 = \operatorname{dim}([0,1] \times \{0\}) \leq \operatorname{dim}(X).\]
\end{exmp}

\begin{rem}
    While Property (SP) is strictly weaker than real rank zero for commutative C*-algebras in most cases, these classes still coincide in a non-trivial setting. 
    It is known that if $G$ is a locally compact abelian group, then $C^*(G)$ has real rank zero if and only if it has Property (SP).
    In particular, this occurs exactly when $G$ is generated by compact elements \cite[Proposition 3.5]{MR3774314}.
\end{rem}

We now move from studying commutative C*-algebras with Property (SP) to the more general setting of commutative SP-inclusions. 
 
 \begin{exmp}\label{exmp:commutative} We shall now exhibit an example of {\bf an SP-inclusion where neither algebra has Property (SP)}. 
    
    Let $Y = \mathcal{C} \times [0,1]$ where $\mathcal{C}$ denotes the Cantor ternary set and let $X = [0,1]$. 
    Let $f: \mathcal{C} \twoheadrightarrow [0,1]$ be the restriction of the Cantor function to $\mathcal{C}$; this function is continuous and surjective.

    Consider the continuous surjection
        \begin{align*}
        \Phi: Y &\twoheadrightarrow X\\
        (x_1, x_2)&\mapsto f(x_1).
    \end{align*}
    We claim that the induced inclusion 
     $$\Phi^*(C(X)) \subseteq  C(Y),$$
    where
    \begin{align*}
        \Phi^*: C(X)&\to C(Y)\\
        g&\mapsto g\circ \Phi, 
    \end{align*}
    is an SP-inclusion. In order to prove this, we in fact show that it is a RR0-inclusion. 
     
    First, notice that $C(X)$ does not contain any non-trivial projections, as $X$ is connected. So, neither $C(X)$ nor $\Phi^*(C(X))\cong C(X)$ have Property (SP), since $\Phi^*$ is a $*$-isomorphism onto its image.  
    We will show that $C(Y)$ does not have Property (SP) by finding a hereditary subalgebra which does not contain any nonzero projections.
    Let $U := \mathcal{C} \times (0,1)$ be an open subset of $Y = \mathcal{C} \times [0,1]$.
    We claim that $U$ does not contain any compact open subsets. 
    Let $$\pi_2: \mathcal{C} \times (0,1) \twoheadrightarrow (0,1)$$ be the projection onto the second coordinate, which is both continuous and open  \cite[Theorem 8.6]{MR2048350}.
    Suppose by way of contradiction that there exists some nonempty, compact, open set $K \subseteq U$.
    By continuity and openness of the projection map, $\pi_2(K)$ would be a nonempty, compact, open set. 
    In particular, $\pi_2(K)$ would be a nonempty clopen subset of $(0,1)$, which contradicts the fact that the latter is connected. 
    Therefore, $C_0(U)$ is a hereditary subalgebra of $C(Y)$ which does not contain any nonzero projections; hence $C(Y)$ does not have Property (SP).
     
    Now let $Y_c$ denote the quotient space of connected components of $Y$.
    In particular, we have that \[Y_c \cong \left\{ \{x_1\} \times [0,1]: x_1 \in \mathcal{C} \right\} \cong \mathcal{C}\] and that the quotient map $q: Y \twoheadrightarrow Y_c$ 
    is a continuous surjection.
    We also define a continuous surjection 
    \begin{align*}
        \pi: Y_c &\twoheadrightarrow X \\
        (x_1,[0,1]) &\mapsto f(x_1).
    \end{align*}
    
    By \cite[Theorem 2.1]{GN25}, $\Phi^*(C(X)) \subseteq  C(Y)$ is a RR0-inclusion if and only if there is an intermediate commutative C*-algebra with totally disconnected spectrum. 
    Since $Y_c \cong \mathcal{C}$ is totally disconnected and $\Phi = \pi \circ q$, our induced inclusion 
    $$
    \Phi^*(C(X)) = q^* (\pi^* (C(X))) \subseteq q^* (C(Y_c)) \subseteq C(Y)
    $$
    is indeed a RR0-inclusion, and therefore an SP-inclusion. 
 \end{exmp}

Finally, let us consider the topological meaning of $C(X) \subseteq C(Y)$ being an SP-inclusion.

 \begin{prop}\label{prop:TopologicalSPInclusion}
     Let $X, Y$ be compact Hausdorff spaces and let $\Phi: Y \twoheadrightarrow X$ be a continuous surjection. The induced inclusion \[\Phi^*( C(X)) \subseteq C(Y)\] is an SP-inclusion if and only if for any nonempty open subset $U \subseteq X,$ the pre-image $\Phi^{-1}[U] \subseteq Y$ contains a nonempty compact open subset.
 \end{prop}

 \begin{proof}
     Suppose that  $\Phi^*( C(X)) \subseteq C(Y)$ is an SP-inclusion and let $U \subseteq X$ be a nonempty open set.
     Consider a function $0 \neq g \in C(X)_+$ with open support $\operatorname{supp}(g) := \{x \in X: g(x) \neq 0\} \subseteq U.$ 
     Since $\Phi^*$ is an injective $\ast$-homomorphism and by our choice of $g$, we have \[0 \neq \Phi^*(g) = g \circ \Phi  \in C(Y)_+.\]
    By our assumption, there exists a nonzero projection 
    \[p \in \overline{(g \circ \Phi)C(Y)(g \circ \Phi)} = C_0\left(\{y \in Y: (g \circ \Phi)(y) \neq 0\}\right).\]
    Moreover, since \[(g \circ \Phi)(y) \neq 0 \iff \Phi(y) \in \operatorname{supp}(g),\] we must have \[\{y \in Y: (g \circ \Phi)(y) \neq 0\} = \Phi^{-1}[\operatorname{supp}(g)] \subseteq \Phi^{-1}[U].\]
    By the usual correspondences, we get that our projection $p$ must be the characteristic function of a nonempty compact open subset $K \subseteq \Phi^{-1}[U] \subseteq Y$, so we are done.

    Conversely, suppose that the pre-image $\Phi^{-1}[U] \subseteq Y$ contains a nonempty compact open subset whenever $U \subseteq X$ is open. 
    We wish to show that \[\Phi^*( C(X)) \subseteq C(Y)\] is an SP-inclusion.

    Take some $0 \neq g \in C(X)_+$ so that $0 \neq g \circ \Phi = f \in \Phi^*(C(X))_+$.
    Then $U := \operatorname{supp}(g) \subseteq X$ is a nonempty open set and hence there exists some nonempty compact open set \[K \subseteq \Phi^{-1}[U] = \Phi^{-1}[\operatorname{supp}(g)] = \{y \in Y: (g \circ \Phi)(y) \neq 0\} = \{y \in Y: f(y) \neq 0\}.\]
    In particular, the characteristic function of $K$ is a nonzero projection $p = \mathbf{1}_K$ with \[ p \in C_0\left(\{y \in Y: f(y) \neq 0\}\right) = \overline{f C(Y) f},\] as required.
 \end{proof}

\section{Permanence Properties}\label{sec:permanenceProperties}

\subsection{Unitization} In the sequel, it will sometimes be easier to prove permanence of SP-inclusions under the assumption that these inclusions are unital. It turns out that we can always pass to the unital case.

\begin{notn}
    Let $A$ be a nonunital C*-algebra. 
    We let $\widetilde{A} := A \oplus \mathbb{C}$ denote the minimal unitization of $A$. 

Abusing this notation, we will sometimes simply write $A$ when we mean the canonical copy of $A$ inside $\widetilde{A}.$
In this way, we view a nonunital C*-algebra $A$ as a norm-closed two-sided ideal in $\widetilde{A}$, noting that $\widetilde{A}/A \cong \mathbb{C}$. Moreover, $A$ is an {\bf essential ideal} of $\widetilde{A}$, meaning that if $J$ is any nonzero ideal of $\widetilde{A},$ then $A \cap J \neq \{0\}.$
Equivalently, this means that for $x \in \widetilde{A}$, $xA = 0$ if and only if $x = 0.$
\end{notn}

First, we will prove that Property (SP) for C*-algebras is permanent under unitization. 

\begin{lem}\label{lem:unitization}
    Let $A$ be a nonunital C*-algebra. Then $A$ has Property (SP) if and only if its minimal unitization $\widetilde{A}$ has Property (SP). 
\end{lem}

\begin{proof}
    First, suppose that $A$ is a nonunital C*-algebra with Property (SP). 
    Let $0 \neq b \in \widetilde{A}_+$. 
    We wish to show that $\overline{b\widetilde{A}b}$ contains a nonzero projection. %
    
    Since $A$ is an essential ideal of $\widetilde{A}$, we have that 
    $$
    H:= \overline{b\widetilde{A}b} \cap A \neq \{0\}.
    $$
    Indeed, since $A\subseteq\widetilde{A}$ is an ideal, then $bAb\subseteq b\widetilde{A}b\cap A,$ and since $A$ is essential, $b\neq 0$ implies $Ab\neq \{0\}$, and in turn $bAb\neq \{0\}.$
    In particular, $H$ is a nonzero hereditary subalgebra of $A$. 
    Since $A$ has Property (SP), there exists a nonzero projection $$p \in H \subseteq \overline{b\widetilde{A}b}$$ as required.
    
    Conversely, let us suppose that $\widetilde{A}$, the minimal unitization of a nonunital C*-algebra $A$, has Property (SP). 
    Since $A$ is a closed two-sided ideal in $\widetilde{A}$, it is a hereditary C*-subalgebra of $\widetilde{A}$ (see e.g., \cite[Proposition II.5.1.1(ii)]{Blackadar}).
    By transitivity, if $H$ is any nonzero hereditary C*-subalgebra of $A$, it is also a hereditary C*-subalgebra of $\widetilde{A}$.
    The fact that $\widetilde{A}$ has Property (SP) means $H$ contains a nonzero projection and so $A$ also has Property (SP). 
\end{proof}
\begin{rem}
    In fact, a more general variant of  Lemma \ref{lem:unitization} holds. 
    If $A\subseteq B$ is an essential ideal, then $A$ has Property (SP) if and only if $B$ has Property (SP), by the same proof. 
    In particular, $A$ has Property (SP) if and only if its multiplier algebra $M(A)$ has Property (SP). 
\end{rem}

We will say that $1 \in A \subseteq B$ is a {\bf unital inclusion} of C*-algebras if both $A, B$ are unital and $1 = 1_A = 1_B$.

\begin{prop}\label{prop:unital}
    Let $A, B$ be nonunital C*-algebras. Then $A \subseteq B$ is an SP-inclusion if and only if $1 \in \widetilde{A} \subseteq \widetilde{B}$ is a unital SP-inclusion. 
\end{prop}

\begin{proof}
    First, suppose that $A$ and $B$ are nonunital C*-algebras such that $A \subseteq B$ is an SP-inclusion.
    We take their minimal unitizations $\widetilde{A}$ and $\widetilde{B}$ respectively, with the same unit so that $1 \in \widetilde{A} \subseteq \widetilde{B}$ is a well-defined unital inclusion.
    Let $0 \neq x \in \widetilde{A}_+$. 
    We wish to show that $\overline{x\widetilde{B}x}$ contains a nonzero projection. 

    First, since $A$ is an  essential ideal in $\widetilde{A}$, we have that
    $xA \neq \{0\},$ so there exists some $y \in A$ such that $xy \neq 0$.
    In particular, $$0 \neq a:= xyy^*x \in A_+$$ and $a \in xAx \subseteq \overline{x\widetilde{B}x}.$ 
    
    Now, by the assumption that $A \subseteq B$ is an SP-inclusion, there exists a nonzero projection $$p \in \overline{aBa} \subseteq \overline{a \widetilde{B}a} \subseteq \overline{x\widetilde{B}x}, $$ where the final containment comes from the fact that $\overline{a \widetilde{B}a}$ is the smallest hereditary subalgebra of $\widetilde{B}$ containing $a$.
    So $\widetilde{A} \subseteq \widetilde{B}$ is an SP-inclusion.

    Conversely, let us suppose that $\widetilde{A} \subseteq \widetilde{B}$ is an SP-inclusion. 
    That is, for any $0 \neq a \in A_+ \subseteq \widetilde{A}_+,$ there is a nonzero projection $p \in \overline{a\widetilde{B}a}.$
    Recall (cf. the proof of Lemma \ref{lem:unitization}) that $B$ is a hereditary subalgebra of $\widetilde{B}$.
    Since \[a \in \overline{aBa} \subseteq \overline{a \widetilde{B}a}\] and we know that $\overline{a\widetilde{B}a}$ is the smallest hereditary subalgebra of $\widetilde{B}$ containing $a$, we must have \[\overline{a\widetilde{B}a} = \overline{aBa}.\] In particular, we have \[ 0 \neq p \in  \overline{aBa},\]
    hence $A \subseteq B$ is an SP-inclusion, as required.
\end{proof}

\begin{rem}
    Proposition \ref{prop:unital} generalizes as follows.
    If $A_0, B_0$ are essential ideals of $A, B$ respectively and the following diagram commutes 

    \begin{center}
    \begin{tikzcd} %
    A_0 \arrow[d, phantom, sloped, "\subset"]\arrow[r, phantom, sloped, "\subseteq"] & B_0\arrow[d, phantom, sloped, "\subset"]\\
    A\arrow[r, phantom, sloped, "\subseteq"] & B
    \end{tikzcd}
    \end{center}
    then $A_0\subseteq B_0$ is an SP-inclusion if and only if $A\subseteq B$ is an SP-inclusion, by the same proof. 
    In particular, we can take multiplier algebras to unitize {\bf nondegenerate} SP-inclusions: if $ \overline{A_0B_0} = \overline{B_0A_0} = B_0$, then $A_0 \subseteq B_0$ is an SP-inclusion if and only if $M(A_0) \subseteq M(B_0)$ is a unital SP-inclusion.
\end{rem}

We would also like to consider the ``mixed'' case of inclusions $A \subseteq B$, where only one of $A, B$ is unital. 

If $A$ is nonunital, $B$ is unital, and $A \subseteq B$, then we can view the unitization $\widetilde{A}$ as a subalgebra of $B$ by identifying $\widetilde{A}$ with $C^*(A, 1_B) \subseteq B$ via $$a + \lambda\cdot 1 \mapsto a + \lambda\cdot 1_B.$$ 
In this case, it is immediate from the proof of Proposition \ref{prop:unital} that $A \subseteq B$ is an SP-inclusion if and only if $1_B \in C^*(A, 1_B) \subseteq B$ is a unital SP-inclusion. 

If $A$ is unital, $B$ is nonunital, and $A \subseteq B$, then we cannot simply pass to $A \subseteq \widetilde{B}$, as the latter is not typically a unital inclusion. 
Instead, notice that we have $1_A \in A \subseteq B$. 

To avoid any implication that $1_A$ is a unit in $B$ (which we emphasize is a nonunital C*-algebra), let us write $p := 1_A \in A \subseteq B$. 
It is clear that $p$ is a projection in $B$ and that $A = pAp \subseteq pBp$ is a unital inclusion. 
For every $0 \neq a \in A_+$, we have $ap = pa = a$ and so for any $b \in B$ we have $$aba = (ap)b(pa) = a(pbp)a.$$
It follows that $$\overline{aBa} = \overline{a(pBp)a}.$$
So, in this case, by arguing as above, we get that $A \subseteq B$ is an SP-inclusion if and only if $p \in A \subseteq pBp$ is a unital SP-inclusion. 

Having exhausted all possible cases, we have shown that we can always assume that our SP-inclusions are unital SP-inclusions, without loss of generality.

\subsection{Hereditary Subalgebras}
We study the interplay between SP-inclusions and intermediate hereditary C*-subalgebras. The statements and proofs from this section closely follow those of 
 \cite[\S3]{GN25} in the context or RR0-inclusions. 

\begin{prop}\label{prop:hereditary}
    Let $A \subseteq B$ be an inclusion of C*-algebras and let $H$ be a hereditary subalgebra of $B$ such that $A \subseteq H$. If $A \subseteq B$ is an SP-inclusion, then so is $A \subseteq H$.
\end{prop}

\begin{proof}
    Take $0 \neq a \in A_+$ so that $\overline{aBa}$ contains a projection.
    It is clear that $\overline{aHa} \subseteq \overline{aBa}$.
    On the other hand, since $\overline{aBa}$ is the smallest hereditary subalgebra of $B$ containing $a$ we have $\overline{aBa} \subseteq \overline{aHa}$. 
    By assumption, this means that there is a nonzero projection $p \in \overline{aBa} = \overline{aHa}.$
    In other words, $A \subseteq H$ is an SP-inclusion.
\end{proof}

\begin{rem}
   The assumption in Proposition \ref{prop:hereditary} that $H\subseteq B$ is hereditary is necessary. 
   Indeed, we have seen SP-inclusions $A\subseteq B,$ where $A$ does not have Property (SP), such as Example \ref{exmp:irrationalRotation}.
   In such cases, $A \subseteq A$ is not an SP-inclusion.
\end{rem}

\begin{cor}
    Let $A \subseteq B$ be an inclusion of C*-algebras and let $0 \neq a \in A_+$. If $A \subseteq B$ is an SP-inclusion, then so is the induced inclusion $\overline{aAa} \subseteq \overline{aBa}$.
\end{cor}

\begin{proof}
    Since $A \subseteq B$ is an SP-inclusion, so is the inclusion $\overline{aAa} \subseteq B$. By Proposition \ref{prop:hereditary}, $\overline{aAa} \subseteq \overline{aBa}$ is an SP-inclusion.
\end{proof}

\subsection{Inclusions and Quotients}
In this section, we see that SP-inclusions are permanent under taking ``larger'' inclusions. 
We will also see a non-permanence result which distinguishes SP-inclusions from RR0-inclusions; namely, we will see that we cannot pass to quotients.

\begin{prop}\label{prop:superalg}
    Let $A \subseteq B$ be an SP-inclusion. If $B \subseteq D$ is any inclusion of C*-algebras, then $A \subseteq D$ is an SP-inclusion.
\end{prop}

\begin{proof}
    This follows from the fact that for any $0 \neq a \in A_+$, there exists a nonzero projection $p \in \overline{aBa} \subseteq \overline{aDa}.$
\end{proof}

To prove the following result, we will use the notion of residual (SP) for C*-algebras. 
We say that a C*-algebra $A$ has {\bf residual (SP)} if $A/I$ has Property (SP) for every closed two-sided ideal $I \subseteq A$. 
This notion was introduced by Pasnicu and Phillips, who showed \cite[Example 7.2]{MR3352760} that Property (SP) is strictly weaker than residual (SP).

\begin{prop}\label{prop:extension}
    Let $A \subseteq B$ be an SP-inclusion. Let $I$ and $J$ be closed two-sided ideals of $A$ and $B$, respectively, with $I \subseteq J$. The induced inclusion $A/I \subseteq B/J$ need not be an SP-inclusion. 
\end{prop}

\begin{proof}
    Let $A$ be any C*-algebra with Property (SP) but not having residual (SP).
    Let $I$ be a closed two-sided ideal of $A$ such that $A/I$ does not have Property (SP). 
    Then $A \subseteq A$ is an SP-inclusion but the induced inclusion $A/I \subseteq A/I$ is not, by Theorem \ref{thm:mainSP}.
\end{proof}

\begin{rem}
    If we replace ``SP-inclusion'' with ``RR0-inclusion'' in the hypotheses of Proposition \ref{prop:extension}, then the induced inclusion \emph{is} a RR0-inclusion \cite[Proposition 3.4]{GN25}, hence it is an SP-inclusion.
\end{rem}

\subsection{Inductive Limits} Whereas many nice examples of C*-algebras with Property (SP) are constructed as inductive limits (e.g., Goodearl algebras \cite{Goodearl} and AF-algebras), it is interesting to consider the permanence of SP-inclusions under such limits. 

\begin{lem}\label{lem:directLimSP}
    Let $(A_\lambda)_{\lambda \in \Lambda}$ be a directed system of C*-algebras with injective $\ast$-homomorphisms $\varphi_{\lambda, \mu}: A_\lambda \to A_\mu$ for $\lambda \leq \mu \in \Lambda$, inductive limit $A = \underrightarrow\lim A_\lambda$, and canonical maps $\varphi_{\lambda}: A_\lambda \to A$. If there exists some $\lambda_0 \in \Lambda$ such that $\phi_{\lambda, \mu} (A_\lambda) \subseteq A_\mu$ is an SP-inclusion for all $\mu > \lambda > \lambda_0$, then $\varphi_{\lambda} (A_\lambda) \subseteq A$ is an SP-inclusion for all $\lambda > \lambda_0$.
\end{lem}

\begin{proof}
     We may assume (by e.g., \cite[Paragraph II.8.2.1]{Blackadar}) that the algebras $A_\lambda$ are all subalgebras of $A$, with $A_\lambda \subseteq A_\mu$ whenever $\lambda \leq \mu \in \Lambda$ and that $\displaystyle A = \overline{\cup_{\lambda \in \Lambda} A_\lambda}$. Then the result follows immediately from Proposition \ref{prop:superalg}. 
\end{proof}

The converse of Lemma \ref{lem:directLimSP} does not hold. 
In particular, given a direct family of C*-algebras algebras $(A_\lambda)_{\lambda \in \Lambda}$ as above,  we can have that $A_{\lambda} \subseteq A$ is an SP-inclusion for all $\lambda \in \Lambda$, but $A_{\lambda} \subseteq A_{\mu}$ is not an SP-inclusion for \emph{any} $\lambda \leq \mu \in \Lambda$.

\begin{exmp}\label{exmp:Goodearl}
    Let $X := [0,1]$, $k_n := (n!)^2$, $\ell_n = 1$, and let $\{q_n\}_{n=1}^{\infty}$ be an enumeration of $\mathbb{Q} \cap [0,1]$. 
    Let $A$ be the {\bf Goodearl algebra} constructed from this data. 
    That is, let $A$ be the inductive limit of the sequence 
    \[ A_{1} := C(X, M_{k_1}) \xrightarrow{\varphi_1} A_2 := C(X, M_{k_2}) \xrightarrow{\varphi_2} \cdots, \]
    where $\varphi_n$ is the unital $\ast$-homomorphism given by
    \[ \varphi_n(a)(x) = \operatorname{diag}(a(q_n), a(x), \ldots, a(x)),\]
    for $x \in [0,1]$ and $a \in C(X, M_{k_n})$. 
     For $n \leq m,$ write $\varphi_{n,m} := \varphi_{m-1} \circ \varphi_{m-2} \circ \cdots \circ \varphi_n$ for the $\ast$-homomorphism that takes a copy of $A_n$ into $A_m$ and write $\varphi_{n,\infty}$ for the map taking a copy of $A_n$ into the inductive limit $A$.
    Since $\cup_{n=k}^\infty \{q_n\}$ is dense in $[0,1]$ for each $k$, $A$ has Property (SP) by \cite[Theorem 5]{Goodearl} and so each 
    $$
    \varphi_{n, \infty}(A_{n}) \subseteq A
    $$
    is an SP-inclusion by Theorem \ref{thm:mainSP}. 

    Fix an $n$ and let us first consider the inclusion 
    $$
    \varphi_n(A_n) \subseteq A_{n+1}.
    $$
    Consider a nonzero function $f \in C([0,1])_+$ such that $f(q_n) = 0$ and $V:= \operatorname{supp}(f) = \{x \in X: f(x) \neq 0\}$ is open in the relative topology of $[0,1]$.  
    Let $a(x) := \operatorname{diag}(f(x), 0, \ldots, 0) \in A_n$.
    Then $0 \neq \varphi_n(a) \in \varphi_n (A_n)_{+}$ and $\operatorname{supp}(\varphi_n(a)) = V$. 
    
    Suppose by way of contradiction that $\varphi_n(A_n) \subseteq A_{n+1}$ is an SP-inclusion.
    In particular, suppose that there exists some $0 \neq p \in \Proj({\overline{\varphi_n(a)A_{n+1}\varphi_n(a)}}).$
    By definition of $a$ and $\varphi_n$, we have $p(x) = 0$ whenever $x \notin V$.
    Now fix some $x_0 \in V$. 
    Let $I \subseteq V$ be the connected component of $x_0$.
    Recall that $I \subseteq [0,1]$ is an interval with $\partial_{[0,1]} I \subseteq [0,1] \setminus V$. 
   Since $x \mapsto \norm{p(x)}$ is a continuous map which takes values in $\{0, 1\}$, it is constant on any connected set.
    In particular, $\norm{p|_I(x)} = 1$ for all $x \in I$,  since we chose $x_0$ so that $p(x_0) \neq 0$.
    Now, suppose that $(x_\lambda) \subseteq I$ is a net converging to $y \in \partial_{[0,1]} I$. 
    Then $\norm{p(x_\lambda)} \to \norm{p(y)}$, but $\norm{p(x_\lambda)} = 1$ for all $x_\lambda \in I$ and $\norm{p(y)} = 0$ since $y \in [0,1] \setminus V$, a contradiction.
    So, $\varphi_n(A_n) \subseteq A_{n+1}$ is not an SP-inclusion.
    
    The fact that 
    $$
    \varphi_{n,m}(A_n) \subseteq A_{m}
    $$
    is not an SP-inclusion for any $m \geq n+1$ follows by a similar argument, taking $0 \neq f \in C([0,1])_+$ such that $f(q_i) = 0$ for each $i \in \{n, \ldots, m-1\}$ and $V:= \operatorname{supp}(f)$ is open.
    
     Moreover, we can show that \emph{none} of the inclusions in this construction are RR0-inclusions.
     
    First, it is immediate that none of the finite-level inclusions are RR0-inclusions, because none of them are SP-inclusions.
     It remains to consider the inclusions
     $$
     \varphi_{n,\infty}(A_n)\subseteq A.
     $$
     Since $X=[0,1]$ is not totally disconnected and \[\gamma := \prod_{j=1}^{\infty}\frac{k_{j+1} - k_j \ell_j}{k_{j+1}} = \prod_{j=1}^{\infty} \frac{\left((j+1)!\right)^2 - (j!)^2}{\left((j+1)!\right)^2} = \prod_{j=1}^{\infty} \left(1 - \frac{1}{(j+1)^2}\right) > 0,\] it follows from \cite[Theorem 6]{Goodearl} that the linear span of projections in $A$ is not dense in $A$ and, as such, $A$ does not have real rank zero. 
     In particular, since $X$ is not totally disconnected, there exist distinct points $y, z \in X$ and a continuous function $g: X \to [0,1]$ such that $g(y)=0$ and $g(z) = 1$.
     Notice that $g\in A_1.$
     In the proof of \cite[Theorem 6]{Goodearl}, Goodearl shows that $\varphi_{1, \infty}(g) \in A$ cannot be approximated within $\gamma/4$ by a linear combination of projections from $A$. 
     In particular, $\varphi_{1, \infty}(g)$ is a self-adjoint element of $\varphi_{1, \infty}(A_1)$ which cannot be approximated arbitrarily well in $A$ by a self-adjoint element with finite spectrum. 
     By \cite[Theorem 1.7]{GN25}, it follows that $\varphi_{1, \infty}(A_1) \subseteq A$ is not a RR0-inclusion. 
     
     But there is nothing special about taking $n=1$ in Goodearl's argument.
     For $n > 1$, define \[\gamma_n := \prod_{j=n}^\infty \frac{k_{j+1} - k_j\ell_j}{k_{j+1}}.\]
     Taking 
     $a := \operatorname{diag}(g, g, \ldots, g) \in A_n$ and following the proof of \cite[Theorem 6]{Goodearl}, we see
     that $\varphi_{n, \infty}(a)$ cannot be approximated within $\gamma_n/4$ by linear combinations of projections in $A$. 
     Since $\gamma_n \geq \gamma > 0$ for all $n > 1, $ then in particular $\varphi_{n, \infty}(a)$ cannot be approximated within $\gamma/4$ by linear combinations of projections in $A$ and so $\varphi_{n, \infty}(A_n) \subseteq A$ is not a RR0-inclusion.
\end{exmp}

Next, we will consider the setting of simultaneous towers of algebras, and establish a sufficient condition to obtain an SP-inclusion at the limits. 
In order to do so, we will need a technical lemma, which will allow us to approximate positive elements in a C*-algebra by positive elements from a dense subalgebra.
The content of the lemma is well-known in the context of inductive limits of C*-algebras, but we were unable to locate a proof in the standard references. 
The result we establish is (formally) more general.

\begin{lem}\label{lem:positiveApprox}
    Let $A_0 \subseteq A$ be an inclusion of C*-algebras, let $0 \neq a \in A_+$, and take $\varepsilon \in (0, \norm{a}]$. If there exists some $0 \neq y \in A_0$ with $\norm{y-a} < \varepsilon$ then there exists $0 \neq x \in (A_0)_+$ such that $\norm{x-a} < 2\varepsilon.$
\end{lem}

\begin{proof}
    Let $A_0 \subseteq A$, $0 \neq a \in A_+$, $\varepsilon \in (0, \norm{a}]$, and $y \in A_0$ be as in the statement of the lemma.

    Define $0 \neq h = \frac{y + y^*}{2} \in (A_0)_{sa}.$
    Since $a \geq 0,$ we have that $h-a = \frac{(y-a)+(y-a)^*}{2}$, hence \[ \norm{h-a} \leq \norm{y-a} < \varepsilon.\]

   Now, recall that because $h$ is self-adjoint, it can be decomposed into its usual positive and negative parts, written $h = h_+ - h_-$.
   We wish to show that $\norm{h_-} \leq \norm{h-a}$. 
   
   If $A$ is not unital, we pass to $\widetilde{A}$, the minimal unitization of $A$ (if $A$ is already unital, observe that the following argument will go through in $A$ instead of $\widetilde{A}$).
    Since $h-a \in \widetilde{A}_{sa}$, we have 
    \begin{align*}
        -\norm{h-a}1_{\widetilde{A}} &\leq h-a \\
        \implies a -\norm{h-a}1_{\widetilde{A}} &\leq h.
    \end{align*}
    Since $a \geq 0$ this means that\[-\norm{h-a}1_{\widetilde{A}} \leq a - \norm{h-a}1_{\widetilde{A}} \leq h.\]
    In particular, we obtain the following bounds on the spectrum of $h$ in $\widetilde{A}$ \[\sigma_{\widetilde{A}}(h) \subseteq [-\norm{h-a}, \norm{h}].\]
    Since $h_-$ is normal and $h_-= f(h)$ where $f(t) = \max\{-t, 0\}$, the continuous functional calculus yields \[\norm{h_-} = \max_{\lambda \in \sigma_{\widetilde{A}}(h)} f(\lambda) = \max_{\lambda \in \sigma_{\widetilde{A}}(h)} \max\{-\lambda, 0\} \leq \norm{h-a}.\]
    
    Finally, let $x := h_+.$
    Then \begin{align*}
        \norm{x-a} = \norm{h_+-a} \leq \norm{h_+-h} + \norm{h-a} 
        &=\norm{h_-} + \norm{h-a} \\
        &\leq 2\norm{h-a} \\
        &< 2\varepsilon,
    \end{align*} as required. 
\end{proof}

\begin{prop}\label{prop:directLimTowers}
    Let $(A_\lambda)_{\lambda \in \Lambda}$ and $(B_\lambda)_{\lambda \in \Lambda}$ be directed systems of C*-algebras with injective $\ast$-homomorphisms $\varphi_{\lambda, \mu}: A_\lambda \to A_\mu$ and $\psi_{\lambda, \mu}: B_\lambda \to B_\mu$ for $\lambda \leq \mu \in \Lambda$, inductive limits $A = \underrightarrow\lim A_\lambda$ and $B = \underrightarrow\lim B_\lambda$, and canonical maps $\varphi_{\lambda}: A_\lambda \to A$, $\psi_\lambda: B_\lambda \to B$ for $\lambda \in \Lambda$. Moreover, suppose that $A_\lambda \subseteq B_\lambda$ and $\psi_\lambda |_{A_\lambda} = \varphi_{\lambda}$ for all $\lambda \in \Lambda$. If there exists some $\lambda_0 \in \Lambda$ such that $A_\lambda \subseteq B_\lambda$ is an SP-inclusion for all $\lambda > \lambda_0$, then $A \subseteq B$ is an SP-inclusion.
\end{prop}

\begin{proof}
    We may assume that the algebras $A_\lambda$ (resp. $B_\lambda$) are all subalgebras of $A$ (resp. $B$), with $A_\lambda \subseteq A_\mu$ (resp. $B_\lambda \subseteq B_\mu$) whenever $\lambda \leq \mu \in \Lambda$ and that $\displaystyle A = \overline{\cup_{\lambda \in \Lambda} A_\lambda}$ (resp. $\displaystyle B = \overline{\cup_{\lambda \in \Lambda} B_\lambda}$).

    First, notice that we must have $A \subseteq B$.
    Now, fix any $0 \neq a \in A_+$. 
    We need to show that there is a nonzero projection in the hereditary subalgebra $\overline{aBa} \subseteq B$. 
    
    For all $\varepsilon \in (0, \norm{a}]$ we can find $\lambda \in \Lambda$ and $0 \neq x \in A_\lambda$ such that $\norm{a-x}< \frac{\varepsilon}{4}$. 
    By Lemma \ref{lem:positiveApprox}, this means that there is some $0 \neq x_\lambda \in (A_\lambda)_+ $ with $\norm{a-x_\lambda} < \frac{\varepsilon}{2}.$
    Moreover, we can insist that $\lambda > \lambda_0$, so that $A_\lambda \subseteq B_\lambda$ is an SP-inclusion and 
    \begin{align*}
        \norm{x_\lambda} &\geq \norm{a} - \norm{a - x_\lambda} > \norm{a} - \frac{\varepsilon}{2} \geq \frac{\norm{a}}{2}.
    \end{align*}
    
    Let $\alpha := \max\left\{\frac{\norm{a}}{\norm{x_\lambda}}-1, 0\right\}$.
    Then any choice of $\delta \in \left(\alpha,\ \frac{\varepsilon}{2\norm{x_\lambda}}\right)$ satisfies $\delta \norm{x_\lambda} < \frac{\varepsilon}{2}$ and $(1+\delta)\norm{x_\lambda} > \norm{a}$. 
    Define $a_{\lambda} := (1+\delta)x_\lambda$. 
    We have that
    \begin{align*}
        \norm{a_\lambda-a} \leq \norm{\delta x_\lambda} + \norm{x_\lambda - a} 
        &< \frac{\varepsilon}{2} + \frac{\varepsilon}{2} = \varepsilon.
    \end{align*}
    Since $a_\lambda$ is positive and $\norm{a_\lambda} > \norm{a} \geq \epsilon$, we get that $(a_\lambda - \varepsilon)_+$ is a nonzero, positive element in $A_\lambda \subseteq B_\lambda$. 
    In particular, there is a nonzero projection $p$ such that
    \[p \in \overline{(a_\lambda - \varepsilon)_+ B_\lambda (a_\lambda - \varepsilon)_+} \subseteq \overline{(a_\lambda - \varepsilon)_+B(a_\lambda - \varepsilon)_+}.\]
    As the latter is $\ast$-isomorphic to a hereditary subalgebra of $\overline{aBa}$ (see \cite[Lemma 4.3]{MR3352760})\footnote{We reproduce the statement here for the reader's convenience: given a C*-algebra $A$, $\varepsilon>0$ and $a,b\in A_+$ with $\|a-b\|<\varepsilon,$ the C*-algebra $\overline{(b-\varepsilon)_+A(b-\varepsilon)_+}$ is isomorphic to some hereditary subalgebra of $\overline{aAa}.$}, there is necessarily a projection in $\overline{aBa}$.
    Since $0 \neq a \in A_+$ was arbitrary, it follows that $A \subseteq B$ is an SP-inclusion.
\end{proof}

\subsection{Stabilization}

Gabe and Neagu left open the question of whether RR0-inclusions pass to matrix amplifications \cite[Question 4.1]{GN25}. 
While we were unable to resolve that question in their setting, we have positive results for our more general inclusions.

\begin{prop}\label{prop:matrixAmp}
If $A \subseteq B$ is an SP-inclusion of C*-algebras, then for any $n \in \mathbb{N}$, so is $M_n(A) \subseteq M_n(B)$.
\end{prop}

\begin{proof}
    Without loss of generality, let us assume that $A \subseteq B$ is a unital SP-inclusion; recall that we may do this by Proposition \ref{prop:unital} and the discussion thereafter.
    
    We wish to show that for any $0 \neq X \in M_n(A)_+$ there is a nonzero projection in $\overline{X M_n(B)X}$.
    For any $1 \leq i,j \leq n$, let $E_{ij}$ denote the matrix with $1$ in the $(i,j)$-th entry and $0$ elsewhere.
    
    First, notice that since $X = (x_{ij})_{i,j = 1}^n \in M_n(A)_+$, each  diagonal entry $x_{jj}$ must be in $A_+$.
    Indeed, since $X = Y^*Y$, we have $$x_{jj} = (Y^*Y)_{jj} = \sum_{i=1}^n y^*_{ij}y_{ij} \in A_+.$$
    If all the diagonal entries $x_{jj}$ were zero, then we would have $$0 = x_{jj} = \sum_{i=1}^n y^*_{ij}y_{ij},$$ hence $y_{ij} = 0$ for all $i, j$.
    In particular, this would mean $Y = 0$ and $X = Y^*Y = 0$.
    Since $X \neq 0$ by assumption, there must be some diagonal entry $x_{kk}$ which is a nonzero element of $A_+$.
    
    Since $A \subseteq B$ is an SP-inclusion and $0 \neq x_{kk} \in A_+$, there exists a nonzero projection $$p \in \overline{x_{kk}Bx_{kk}} \subseteq B.$$
    We will view $p$ as an element of $M_n(B)$ by identifying it with $P := E_{kk}pE_{kk}$.

    Let $X_{kk} := E_{kk}XE_{kk}$. 
    Then $$P \in \overline{X_{kk}M_n(B)X_{kk}},$$ because
    $p \in \overline{x_{kk} B x_{kk}}$ and the canonical identification of $B$ with $E_{kk}M_n(B)E_{kk}$ gives 
  $$  P \in E_{kk}\overline{X_{kk} M_n(B)X_{kk}}E_{kk} =  \overline{E_{kk}X E_{kk}M_n(B) E_{kk}XE_{kk}} = \overline{X_{kk} M_n(B) X_{kk}}.$$
    
    Now, let $T := X^{1/2}E_{kk} \in M_n(B)$ so that $X_{kk} = T^*T$ and $TT^* = X^{1/2}E_{kk}X^{1/2}.$
    In particular, it follows from \cite[Lemma 3.8]{MR3352760} that \[\overline{T^*T M_n(B) T^*T} \cong \overline{TT^* M_n(B) TT^*}.\]
    Since $\overline{T^*T M_n(B) T^*T}$ contains a nonzero projection $P$, there must be some nonzero projection $Q \in \overline{TT^* M_n(B) TT^*}.$ 
    Moreover, since it is clear that $TT^* \in \overline{XM_n(B)X}$, we must have 
    \[
    Q \in \overline{TT^* M_n(B) TT^*} \subseteq \overline{XM_n(B)X},
    \] 
    recalling that $\overline{TT^* M_n(B) TT^*}$ is the smallest hereditary C*-subalgebra of $M_n(B)$ that contains $TT^*.$ 
    In particular, the latter contains a nonzero projection, as required.
\end{proof}

Combining the above result with Proposition \ref{prop:directLimTowers}, we obtain the following, where $\mathcal{K}$ denotes the C*-algebra of compact operators on $\ell^2(\mathbb{N})$.

\begin{cor}\label{cor:stabilization}
    If $A \subseteq B$ is an SP-inclusion of C*-algebras, then so are the inclusions $$A \otimes \mathcal{K} \subseteq B \otimes \mathcal{K},$$
    and 
    $$A \otimes \mathcal{M} \subseteq B \otimes \mathcal{M}$$
    for any UHF algebra $\mathcal{M}$.
\end{cor}

\begin{rem}
    The results in this section could (alternatively) follow immediately by using Theorem \ref{thm:tensor} below.
    However, the matrix amplification result established in Proposition \ref{prop:matrixAmp} highlights the fact that certain permanence properties can be proven for SP-inclusions, while they remain elusive for RR0-inclusions.
    We also find the natural use of our earlier inductive limit results to obtain Corollary \ref{cor:stabilization} appealing.
\end{rem}

\subsection{Tensor Products by C*-algebras with Property (SP)} In this subsection, we are interested in permanence of SP-inclusions under the minimal (spatial) tensor product. For ease of notation, we will simply denote this by $\otimes$.

We will make use of the following result, which was also used in \cite{MR4202393} to show that Property (SP) for C*-algebras is preserved under minimal tensor products.

\begin{lem}\cite[Lemma 2.15]{MR2032998}\label{lem:sliceMap}
    If $A$ and $B$ are C*-algebras and if $H$ is a nonzero hereditary subalgebra of $A \otimes B$, then there exists $0 \neq z \in A \otimes B$ with $zz^{*} \in H$ and $z^*z = a \otimes b$ for some nonzero $a \in A_{+}$ and nonzero $b \in B_{+}$.
\end{lem}

\begin{thm}\label{thm:tensor}
    Let $A \subseteq B$ be an SP-inclusion of C*-algebras. If $D$ is a C*-algebra with Property (SP) then $A \otimes D \subseteq B \otimes D$ is an SP-inclusion. 
\end{thm}
\begin{proof}
    We wish to show that for any $0 \neq x \in (A \otimes D)_+$, the hereditary subalgebra $ \overline{x(B \otimes D)x}$ contains a nonzero projection from $B \otimes D$.
    
    Fix an arbitrary $0 \neq x \in (A \otimes D)_+$ and let $H := \overline{x(A \otimes D)x}$. 
    (Notice that we are first considering a hereditary subalgebra of $A \otimes D$ rather than $B \otimes D$).
    By Lemma \ref{lem:sliceMap}, there exists $0 \neq z \in A \otimes D$ with $h: = zz^{*}\in H$ and $z^{*}z = a \otimes d$ for some $0 \neq a \in A_{+}$ and $0\neq d \in D_{+}$.
    Since $D$ has Property (SP), there is a nonzero projection $q \in \overline{dDd}$.
    Moreover, since $A \subseteq B$ is an SP-inclusion, there is a nonzero projection $p \in \overline{aBa}$.
    In particular, 
    \begin{align*}
        0 \neq p \otimes q \in \overline{aBa} \otimes \overline{dDd} \cong \overline{(a \otimes d)(B \otimes D)(a \otimes d)} &= \overline{z^*z(B \otimes D)z^*z} \\ &\cong \overline{zz^*(B \otimes D) zz^*},
    \end{align*}
    where the final $\ast$-isomorphism follows from \cite[Lemma 3.8]{MR3352760}.
    Finally, since $h \in \overline{x(A \otimes D)x} \subseteq \overline{x(B \otimes D)x}$ we get $$p \otimes q \in \overline{zz^*(B \otimes D) zz^*} = \overline{h(B \otimes D)h} \subseteq \overline{x(B \otimes D)x},$$ as required.
\end{proof}

\begin{rem}
    The analogous statement of Theorem \ref{thm:tensor} does not hold for RR0-inclusions. For example, if $H$ is a separable, infinite-dimensional Hilbert space, then $B(H)$ has real rank zero, but it is well-known that $B(H) \otimes B(H)$ does not \cite[Corollary 1.2]{MR1707747}. In particular, by \cite[Lemma 1.9]{GN25}, $B(H) \subseteq B(H)$ is a RR0-inclusion but $B(H) \otimes B(H) \subseteq B(H) \otimes B(H)$ is not. 
\end{rem}

Until now, we have been developing the theory of ``relative'' Property (SP) for inclusions of C*-algebras.
We now turn to the distinct but related study of the permanence of Property (SP) for C*-algebras under various kinds of inclusions. 

\section{Property (SP) Under Quantum Symmetries}\label{sec:Symmetries}  

In this section, we are interested in answering the question: ``If $A$ has Property (SP), under what kind of inclusions $A \subseteq B$ does $B$ necessarily have Property (SP)?''
This is a rather classical problem, which has been considered on numerous occasions (e.g., \cite{MR1623278}, \cite{MR1862184}, \cite{MR3352760}, \cite{MR5001191}).
In particular, Property (SP) for C*-algebras is known to be permanent under reduced crossed products coming from certain ``nice'' group actions; these are examples of \emph{classical symmetries}. 
\begin{exmp}\cite[Theorem 4.2]{MR1623278} \label{exmp:finiteGroup}
    Let $A$ be a simple (unital) C*-algebra with Property (SP). If $G$ is finite group and $\alpha: G \to \operatorname{Aut}(A)$ is any action by automorphisms, then $A \rtimes_{r, \alpha} G$ has Property (SP). 
    
    Reduced crossed products coming from finite group actions give rise to a canonical faithful conditional expectation $E: A \rtimes_{r, \alpha} G \twoheadrightarrow A$ of \emph{finite Watatani index}. 
    In fact, by \cite[Theorem 5.1]{MR1862184}, Property (SP) for simple (unital) C*-algebras is known to be permanent under any (unital) inclusion $A \subseteq B$ equipped with a faithful conditional expectation of finite index.
\end{exmp}

In Section \ref{subsec:DiscreteInclusions} we will extend this type of result to show that, under mild hypotheses, Property (SP) is also permanent under \emph{quantum symmetries}. 
First, we will need to clarify some terminology and examples from the classical case; this will also allow us to upgrade some known results from the setting of Property (SP) for C*-algebras to the setting of relative Property (SP) for inclusions.

\subsection{Property (SP) and Conditional Expectations}\label{subsec:outerSP} 
We will mainly be interested in the permanence of Property (SP) under inclusions equipped with a (faithful) conditional expectation satisfying suitable properties.

\smallskip
Let $ A \subseteq B$ be an inclusion of C*-algebras and let $E: B \twoheadrightarrow A$ be a conditional expectation. We say that $E$ has the {\bf pinching property} if for any nonzero $x \in B_+$ and any $\varepsilon > 0,$ there is an element $y \in A$ such that
\begin{align}\label{eq:pinching}
    \norm{y^*(x-E(x))y} < \varepsilon \quad \text{ and } \quad \norm{y^*E(x)y} \geq \norm{E(x)} - \varepsilon.
    \tag{$\ast$}
\end{align}
We highlight that we have defined the pinching property without reference to units, in contrast with the conventions from \cite{MR1862184, MR4599249}, as we shall make no use of this assumption.

\smallskip
In order to put the results of this section in its proper context, we include the statement of a theorem from which we drew inspiration. 
    
\begin{lem}\cite[Theorem 2.1]{MR1862184}\label{thm:Osaka1}
    Let $1 \in A \subseteq B$ be a unital inclusion of C*-algebras equipped with a faithful conditional expectation $E: B \twoheadrightarrow A$. Suppose that $A$ has Property (SP). If $E$ has the pinching property, then $B$ has Property (SP). Moreover, every nonzero hereditary C*-subalgebra of $B$ has a projection that is equivalent to some projection in $A$.
\end{lem}

We now show a strengthened version of Lemma \ref{thm:Osaka1}. In particular, we show that in the presence of a faithful conditional expectation with the pinching property, we may weaken the assumption that $A$ has Property (SP) to the assumption that $A \subseteq B$ is an SP-inclusion. Moreover, we remove the assumption of unitality. 
\begin{prop}\label{prop:pinchingSP}
    Let $A \subseteq B$ be an SP-inclusion and let $E: B \twoheadrightarrow A$ be a faithful conditional expectation. If $E$ has the pinching property, then $B$ has Property (SP).
\end{prop}

\noindent The proof of this result is morally the same as the proof given by Osaka in \cite[Theorem 2.1]{MR1862184}. We include the details for the convenience of the reader. 

\begin{proof}
    First let us fix arbitrary $0 \neq x \in B_+$.
    Since $E$ has the pinching property, for any $\varepsilon > 0$, we can find an element $y \in A$ such that $\norm{y^*(x-E(x))y} < \varepsilon $. 
    In particular, pick $\varepsilon \in \left(0, \frac{\norm{E(x)}}{2}\right)$. By faithfulness of the expectation and the second inequality in the pinching property, we have 
    \[\norm{y^*E(x)y} \geq \norm{E(x)} - \varepsilon > \frac{\norm{E(x)}}{2} > 0,\] hence $0 \neq (y^*E(x)y - \varepsilon)_+ \in A^+$.
    Since $A \subseteq B$ is an SP-inclusion, there is some nonzero projection \[p \in  \overline{(y^*E(x)y - \varepsilon)_+ B(y^*E(x)y - \varepsilon)_+}.\]
   By \cite[Lemma 4.3]{MR3352760}, the algebra $\overline{(y^*E(x)y - \varepsilon)_+ B(y^*E(x)y - \varepsilon)_+}$ is $\ast$-isomorphic to a hereditary subalgebra of $\overline{y^*xy B y^*xy}$.
   In particular, there is a nonzero projection $q \in \overline{y^*xy B y^*xy}$.

    Let $z := x^{1/2}y$, so that $z^*z = y^*xy$ and $zz^* = x^{1/2}yy^*x^{1/2}$.
    Then by \cite[Lemma 3.8]{MR3352760}, we have
    \[q \in \overline{z^*z B z^*z} \cong \overline{zz^*Bzz^*} \subseteq \overline{xBx}.\]
    So $\overline{xBx}$ contains a nonzero projection for any nonzero $x \in B_+$.
    In other words, $B$ has Property (SP).    
\end{proof}

\begin{thm}[Theorem \ref{mainthm:pinchingMain}]\label{thm:pinchingMain}    Let $A \subseteq B$ be an inclusion of C*-algebras with a faithful conditional expectation $E:B\twoheadrightarrow A$ which has the pinching property. Then $A \subseteq B$ is an SP-inclusion if and only if $B$ has Property (SP). 
\end{thm}

\begin{proof}
    If $B$ has Property (SP), then $A \subseteq B$ is an SP-inclusion by Theorem \ref{thm:mainSP}.
    On the other hand, if $A \subseteq B$ is an SP-inclusion and $E: B \twoheadrightarrow A$ has the pinching property, then $B$ has Property (SP) by Proposition \ref{prop:pinchingSP}.
\end{proof}

\begin{rem}
    Theorem \ref{thm:pinchingMain} can be used as an obstruction for an SP-inclusion $A\subseteq B$ to have a faithful pinching expectation. 
    Viewing Example \ref{exmp:commutative} under this lens, we conclude no such $E$ exists. 
\end{rem}

Now, we consider a related condition on a conditional expectation $E: B \twoheadrightarrow A.$ 
This condition will be satisfied in many ``canonical'' kinds of inclusions, allowing us to obtain concrete examples. 

\smallskip
Let $A \subseteq B$ be an inclusion of C*-algebras and let $E: B \twoheadrightarrow A$ be a conditional expectation. Following \cite[Definition 2.2]{MR1862184}, we say that $E$ is an {\bf outer conditional expectation} if for any element $x \in B$ with $E(x) = 0$ and any nonzero hereditary C*-subalgebra $H \subseteq A$, \[ \inf\left\{\norm{hxh}: h \in H_+, \norm{h} = 1\right\} = 0.\]
We make no reference to units when defining outer expectations, as we will make no use of them in the following.

The proof of the following fact is sketched as part of the proof of  \cite[Corollary 2.3]{MR1862184}. 
We include a full proof here for the benefit of the reader.

\begin{prop}\label{prop:pinchingOuter}
    Let $A \subseteq B$ be an inclusion of C*-algebras. If $E: B \twoheadrightarrow A$ is an outer conditional expectation, then it has the pinching property.
\end{prop}

\begin{proof}
    Let us fix arbitrary $0 \neq x \in B_+$ and $\varepsilon > 0.$

    If $E(x) = 0$, then by outerness of the expectation, in any nonzero hereditary subalgebra of $A$, there exists a positive element $0 \neq y$ with $\norm{y} =1$, such that \[ \norm{y(x-E(x))y} = \norm{yxy}  < \varepsilon.\] 
    Moreover it is clear that \[\norm{yE(x)y} = 0 \geq \norm{E(x)} - \varepsilon,\] as required.
    So, both conditions of Eq. (\ref{eq:pinching}) are satisfied.
    
    Now suppose that $E(x) \neq 0$ for our fixed $x$.
     First, notice that for $z:= x - E(x),$ we have $E(z) = 0$.
     So by outerness of the expectation, in any nonzero hereditary subalgebra of $A$, there exists a positive element $0 \neq y$ with $\norm{y} =1$, such that 
     \begin{equation}\label{eq:pinching1}
         \norm{y(x-E(x))y} = \norm{yzy}  < \varepsilon.
         \tag{1}
     \end{equation}
    
    For the purpose of easing notation, let us now write $0 \neq a:= E(x)$ and $M:= \norm{a} > 0.$ 
    Fix some $0 < \delta < \min\{M, \varepsilon\}$. 
    We define a continuous function ${f: [0, \norm{x}] \to [0, \infty)}$ given by
      \[ f(t) = \begin{cases}
        1 \quad & t \geq M \\
        \text{linear} & M - \delta \leq t \leq M \\
        0 & t < M - \delta.
    \end{cases}\]
    
    Then $0 \neq f(a) \in A_+$ and  $H:= \overline{f(a)Af(a)}$ is a nonzero hereditary subalgebra of $A$.
    As described in the proof of Lemma \ref{lem:equivalentSP}, there is a unique projection $q \in A^{**}$ such that $H = qA^{**}q \cap A,$ namely \[q = s(f(a)) = \mathbf{1}_{(M- \delta, \infty)}(a) = \mathbf{1}_{(M-\delta, M]}(a).\]
    Pick any $0 \neq h \in H_+$ with $\norm{h} = 1$.
    Then, since $qh = h = hq$, we get 
    \begin{align*}
        hah = h(qaq)h &\geq h((M-\delta)q)h \\
        & = (M-\delta)h^2.
    \end{align*}
    Taking norms, we get
    \[\norm{hah} \geq (M-\delta)\norm{h^2} = M-\delta > M - \varepsilon.\]
    In our notation, this means 
    \begin{equation}\label{eq:pinching2}
        \norm{hE(x)h} > \norm{E(x)} - \varepsilon.
        \tag{2}
    \end{equation}

    Since (1) holds for \emph{some} $0 \neq y \in H_+$ with $\norm{y} = 1$ and (2) holds for \emph{all} $0 \neq h \in H_+$ with $\norm{h} = 1,$ both equations can be satisfied simultaneously by some element of $H$. 
    In particular, the conditions of Eq. (\ref{eq:pinching}) are satisfied, so $E$ has the pinching property.
\end{proof}

\begin{rem}
    It follows immediately that analogous results to Lemma \ref{thm:Osaka1}, Proposition \ref{prop:pinchingSP}, and Theorem \ref{thm:pinchingMain} hold when we replace ``$E$ has the pinching property'' with ``$E$ is outer''.
\end{rem}

\begin{exmp}\cite[Example 2.4]{MR1862184} \label{exmp:discreteGroup}
    Let $G$ be a discrete group and let $A$ be a (unital) simple C*-algebra. If $\alpha: G \curvearrowright A$ is outer, then the canonical conditional expectation $E: A \rtimes_{r, \alpha} G\twoheadrightarrow A$ is outer. In particular, if $A$ has Property (SP) then so does $A \rtimes_{r, \alpha} G.$

    We note that, as written, \cite[Example 2.4]{MR1862184} requires that $A$ is unital. 
    However, this is not necessary; the proof relies mainly on \cite[Lemma 3.2]{MR634163}, in which no such assumption is made. 
\end{exmp}
We shall next establish similar conclusions in the opposite case, where $A$ is commutative. 
This will also be useful in Section \ref{sec:Boundaries}. 

\smallskip
For a C*-algebra $A$, we say that $\alpha \in \operatorname{Aut}(A)$ is {\bf properly outer} if for any nonzero invariant ideal $I \subseteq A$ and any $\beta \in \operatorname{Inn}(I)$, $\norm{\alpha|_I - \beta} = 2.$
If $\alpha: G \curvearrowright A$ and $\alpha_g$ is properly outer for all $g \neq e$, we just say that $\alpha$ is properly outer. 
If $G$ is discrete and $A = C_0(X)$ then $\alpha: G \curvearrowright A$ is properly outer if and only if it is {\bf topologically free}, meaning that for any $g_1, \ldots, g_n \in G \setminus\{e\},$ the set $\cap_{i=1}^{n} \{x \in X: g_ix \neq x\}$ is dense in $X$.
This is further equivalent to $\operatorname{int}(\{x \in X: gx = x\}) = \emptyset$ for all $g \in G \setminus \{e\}$ (see \cite[Remarks]{MR1258035}).

In the sequel, we will pass between the various characterizations of proper outerness / topological freeness without comment.  

\begin{lem}\label{lem:properlyOuter}
    Let $G$ be a discrete group and let $A := C_0(X)$ be a commutative C*-algebra. If $\alpha: G \curvearrowright A$ is properly outer, then the canonical conditional expectation $E: A \rtimes_{r, \alpha} G \twoheadrightarrow A$ is outer. 
\end{lem}

\begin{proof}
 Let $u_g, g \in G$ be the standard unitaries (in $M(A \rtimes_{r, \alpha} G)$, the multiplier algebra of the reduced crossed product) which implement the properly outer action $\alpha$.
 Given $x \in A \rtimes_{r, \alpha} G$ with $E(x) = 0$, we approximate $x$ by $y = \sum_{i=1}^{n} a_{i} u_{g_i}$ where $a_i \in C_0(X)$ and $g_1, \ldots, g_n$ are distinct elements in $G \setminus \{e\}.$ 
 Here, the fact that we can take finite approximations by elements in $G \setminus \{e\}$ is because $E(x) = 0$ (if $z$ is a finite approximation of $x$ which has an identity term, then we replace it by $y = z - E(z)$, which does not).

 Let $H$ be an arbitrary nonzero hereditary subalgebra of $C_0(X)$, which we can identify with $C_0(U)$ where $U \subseteq X$ is a nonempty open subset of $X$.
 Since $A$ is commutative, proper outerness of $\alpha$ means exactly that the set ${\cap_{i=1}^{n} \{x \in X: g_ix \neq x\}}$ is dense in $X$.
 In particular, 
 $$Y:= U \cap \left(\cap_{i=1}^{n} \{x \in X: g_ix \neq x\}\right) \neq \emptyset.$$
 For any $v \in Y$, we can find a nonempty open neighbourhood of $v$, say $ V \subseteq U$, such that $V \cap g_iV = \emptyset$ for each $i = 1, \ldots, n$.\footnote{
We illustrate how to find this open set when $n=1$. By topological freeness, there is a point $v\in Y$ such that $gv\neq v.$ Since $X$ is Hausdorff, there are disjoint open subsets $V_1, V_1'\subseteq U$ such that $v\in V_1$ and $gv\in V_1'.$ Now consider $V:=V_1\cap g^{-1}V_1'$ and notice $v\in V$ is thus nonempty and open. 
Finally, $V\cap gV= V_1\cap g^{-1}V_1'\cap gV_1\cap V_1'\subseteq V_1\cap V_1'=\emptyset.$
 }  
 
 Next, pick $h \in H_+$ with $\norm{h} = 1$ and $\operatorname{supp}(h) \subseteq V$. For each $g_i$, the covariance relation means that $h a_i u_{g_i}h = h a_i \alpha_{g_i}(h)u_{g_i}$.
 Moreover, the supports of $h$ and $\alpha_{g_i}(h)$ are disjoint because $\operatorname{supp}(h) \subseteq V$ and $\operatorname{supp}(\alpha_{g_i}(h)) \subseteq g_iV$. 
 Therefore, $h a_i \alpha_{g_i}(h)u_{g_i} = 0$ for each $i = 1, \ldots, n$ so that 
 $$\norm{h y h} \leq \sum_{i=1}^n \norm{h a_i u_{g_i} h} = \sum_{i=1}^n \norm{h a_i \alpha_{g_i}(h) u_{g_i}} =0.$$
 Since $\norm{hxh} = \norm{h(x-y)h} \leq \norm{x-y}$ and elements of the form $y$ approximate $x$ arbitrarily well, we get  \[ \inf\left\{\norm{hxh}: h \in H_+, \norm{h} = 1\right\} = 0.\]
 That is, the conditional expectation $E: A \rtimes_{r, \alpha} G \twoheadrightarrow A$ is outer, as required.
\end{proof}

\begin{rem}\label{rem:partialAction}
    A related result has been shown in the more general setting where $X$ is a locally compact Hausdorff space, $G$ is a discrete group, $\alpha$ is a \emph{partial action} of $G$ on $C_0(X)$, and $C_0(X) \rtimes_{r, \alpha} G$ is the (partial) reduced crossed product. 
   In our terminology, \cite[Proposition 2.4]{MR1905819} shows that if $\alpha$ is topologically free, then the canonical faithful conditional expectation $E: C_0(X) \rtimes_{r, \alpha} G \twoheadrightarrow C_0(X)$ has the pinching property. 
\end{rem}

\begin{exmp}\label{exmp:properlyOuter}
     Let $G$ be a discrete group, let $A := C_0(X)$ be a commutative C*-algebra, and let $\alpha: G \curvearrowright C_0(X)$ be topologically free (properly outer). If $C_0(X)$ has Property (SP) then so does $C_0(X) \rtimes_{r, \alpha} G.$ 

     By Remark \ref{rem:partialAction}, an analogous result also whenever $\alpha$ is a topologically free partial action of $G$ on $A$.
\end{exmp}

\subsection{Property (SP) and Discrete Inclusions}\label{subsec:DiscreteInclusions}
In this section, we shall explore the permanence of Property (SP) under \emph{discrete} inclusions of C*-algebras $A\subseteq B,$ which come equipped with a faithful expectation $E:B\twoheadrightarrow A$. 
Roughly speaking, discrete inclusions are meant to posses similar features to those arising as reduced crossed products by actions of discrete (quantum) groups. 

Indeed, the class of unital {\bf irreducible} (i.e. $A'\cap B\cong \mathbb{C}$) discrete inclusions was defined and characterized in \cite{MR5051789} as those inclusions arising as reduced crossed products by outer actions of \emph{unitary tensor categories} (UTC) along with algebraic data, and further studied in \cite{MR5042209, HHP25, HPN2}, and \cite{MR5099905} in the nonunital simple separable setting. 

This class is rather broad with many known examples arising from various different sources, including all finite \emph{Jones-Watatani index} inclusions, reduced crossed products by outer actions of \emph{discrete (quantum) groups}, and core inclusions from Cuntz algebras, among others.  
As such, discrete inclusions can be seen as C*-dynamical systems governed by a quantum symmetry, extending the classical theory involving groups.

\smallskip

We briefly gather the fundamental ingredients herein for the sake of self-containment and the reader's convenience. 
Further details should be pursued in the references listed above. 

Consider a fixed nondegenerate inclusion of $\sigma$-unital C*-algebras $A \subseteq B$ (i.e., both $AB, BA\subseteq B$ are dense) equipped with $E: B \twoheadrightarrow A,$ a faithful conditional expectation onto $A$. 
Associated to this data is the generalized GNS-construction
$$
\mathcal{E}:=L^2(B,E),
$$
which is the completion of $B$ in the $A$-valued norm induced by the $A$-valued inner product 
$$
\langle b\ |\ b'\rangle_A:= E(b^*b').
$$
Naturally, $\mathcal{E}$ becomes a right $B$-$A$ correspondence with this inner product and the obvious and right-$A$ action and adjointable left-$B$  action determined by right/left multiplication. 
Since $E$ is assumed faithful, there is an injective contractive map $B$-$A$ bimodular map
$$
\eta:B\to \mathcal{E},
$$
whose range is dense in $\mathcal{E}.$ 
Whenever the inclusion is unital, $\eta$ is determined by $\Omega:=\eta(1),$ and its range is the dense subspace $\eta(B) =B\Omega\subseteq \mathcal{E}.$

The {\bf projective quasi-normalizer} of the inclusion is given by 
\begin{align*}
    PQN(A\subseteq B):=\{b\in B\ |\ \exists K \in \mathsf{Bim}_{\mathsf{f}}(A),\ \eta(b)\in K\subseteq\eta(B)\}.
\end{align*}
Here, $\mathsf{Bim}_{\mathsf{f}}(A)$ denotes the UTC of \emph{dualizable} bimodules over $A,$ which coincide with those with a finite index in the sense of \cite{MR2085108}.  
In the unital setting, dualizable bimodules coincide with the finitely generated projective bimodules.  
That is, $K\in \mathsf{Bim}_{\mathsf{f}}(A)$ if viewed as a right (respectively, left) $A$ module, it is isomorphic to a complementable submodule of $A^{\oplus n}$ for some $n$ in $\mathbb{N}.$ 
By economy of language, we omitted $E$ from the definition of the projective quasi-normalizer, as it will always be inferred from context. 
It was shown in the unital case in \cite{MR5051789} that 
$$
A\subseteq PQN(A\subseteq B)\subseteq B
$$
is in fact a $*$-subalgebra. 
We say that $A\overset{E}{\subseteq} B$ is {\bf C*-discrete} if $PQN(A\subseteq B)$ is norm-dense in $B$. 

\begin{exmp}
    For each of the C*-algebras $A$, discrete groups $G$, and ``nice'' faithful conditional expectations $E: (A \rtimes_r G) \twoheadrightarrow A$ described in Examples \ref{exmp:finiteGroup}, \ref{exmp:discreteGroup}, and \ref{exmp:properlyOuter}, the inclusion \[ A \overset{E}\subseteq A \rtimes_r G\] is C*-discrete. 
\end{exmp}

\begin{prop}\label{prop:discreteOuterExp}
    Let $A \overset{E}\subseteq B$ be an irreducible C*-discrete inclusion.
    If $A$ is simple and separable, then $E$ is outer.
\end{prop}

\begin{proof}
    Let $b \in PQN(A \subseteq B)$. 
    By \cite[Lemma 4.13 (3)]{MR5099905}
    for any $\varepsilon > 0$ and any nonzero hereditary subalgebra $H \subseteq A$, there exists an element $h \in H_+$ with $\norm{h} = 1$ such that \[\norm{h(b-E(b))h} < \varepsilon.\] 
    In particular, if $E(b) = 0$, then
        \[\norm{hbh} < \varepsilon. \]
    That is, the restriction of the conditional expectation to $PQN(A \subseteq B)$ is outer.
    
    Since the inclusion $A\subseteq B$ is C*-discrete, $PQN(A \subseteq B)$ is dense in $B$. 
    In particular, for any $x \in B$ with $E(x) = 0$ and any $\varepsilon > 0$ we can find $b \in PQN(A \subseteq B)$ such that $\norm{x-b} < \varepsilon/4$.
    
    We have $b-E(b) \in PQN(A \subseteq B)$ and $E(b-E(b)) = 0,$ so we can find $h \in H_+$ with $\norm{h} = 1$ such that $\norm{h(b-E(b))h} < \varepsilon/2$.
    Moreover, \[\norm{x-(b-E(b))} \leq \norm{x-b} + \norm{E(b-x)} \leq \frac{\varepsilon}{2},\] which implies that
    \[ \norm{h(x-E(x))h} = \norm{hxh} \leq \norm{h(x-(b-E(b)))h} + \norm{h(b-E(b))h} < \varepsilon.\] 
    In particular, \[ \inf\left\{\norm{hxh}: h \in H_+, \norm{h} = 1\right\} = 0, \] as required.
    \end{proof}

With this established, we can now call on the general theory of permanence of Property (SP) under outer expectations, as established in Section \ref{subsec:outerSP}, to prove our next main result. 
In particular, we show that Property (SP) is preserved under those quantum symmetries which can be described by actions of UTCs. 

\begin{thm}[Theorem \ref{mainthm:discreteSP}]\label{thm:discreteSP}
    Let $ A \overset{E}\subseteq B$ be an irreducible C*-discrete inclusion.
    If $A$ is simple, separable, and has Property (SP), then $B$ has Property (SP). Moreover, if $D$ is an intermediate subalgebra of the inclusion, that is if $A \subseteq D \subseteq B$, then $D$ also has Property (SP).
\end{thm}

\begin{proof}
   If $A$ has Property (SP) then $A \subseteq B$ is an SP-inclusion, so the fact that $B$ has Property (SP) follows from Propositions \ref{prop:pinchingOuter}, \ref{prop:pinchingSP}, and \ref{prop:discreteOuterExp}. 
   The fact that \emph{any} intermediate subalgebra $D$ also has Property (SP) follows in the same way, after observing that outerness of the expectation is preserved when we restrict to intermediate subalgebras. 
  That is, for any intermediate subalgebra $D$ of our inclusion $A \subseteq B$, the conditional expectation $E_D: D \twoheadrightarrow A$ given by $E_D = E|_D$ is both faithful and outer. 
\end{proof}

Using Theorem \ref{thm:discreteSP}, we can describe quantum analogues of Example \ref{exmp:discreteGroup}.
We include the following as an illustrative example and refer the reader to consult the literature on the operator algebraic approach to locally compact quantum groups (see e.g., \cite{MR3204665}, \cite{MR3675047}) for further details on the C*-algebras involved in the inclusion.
In Appendix \ref{sec:appendix}, we outline the argument that the inclusion described below is indeed irreducible and C*-discrete.  

\begin{exmp}\label{exmp:discreteQuantum}
    Let $\G = (C(\G), \Delta)$ be a (reduced) \emph{compact quantum group}, let $\widehat{\G} = (c_0(\widehat{\G}), \widehat{\Delta})$ be the dual \emph{discrete quantum group}. In particular, we have that
    \begin{equation*}
        c_0(\widehat{\G}) = \bigoplus^{c_0}_{u \in \mathsf{Irr}(\G)} B(H_u)
    \end{equation*} where $\mathsf{Irr}(\G)$ refers to a fixed set of representatives of the equivalence classes of irreducible unitary representations of $\G$ and $B(H_u) \cong M_{n_u}(\mathbb{C})$ for $n_u := \dim H_u.$ 
    Let $A$ be a unital simple C*-algebra, let $\alpha: A \to M(c_0(\widehat{\G}) \otimes_{\min} A)$ be a \emph{left action} of $\widehat{\G}$ on $A$ that is an injective unital $\ast$-homomorphism satisfying 
    \begin{enumerate}[(i)]
       \item $(\widehat{\Delta} \otimes \operatorname{id})\alpha = (\operatorname{id} \otimes \alpha)\alpha$, and
       \item  $\overline{\operatorname{span}}^{\norm{\cdot}}\left\{(x \otimes 1)\alpha(a): x \in c_0(\widehat{\G}), a \in A\right\} = c_0(\widehat{\G}) \otimes_{\min} A$,
    \end{enumerate}
    and let $B :=  A \rtimes_{r, \alpha} \widehat{\G} = C^*(\alpha(A), C(\G) \otimes 1) \subseteq M(K(\ell^2(\widehat{\G})) \otimes A)$ be the reduced crossed product induced by the action.
    
    Suppose that $\alpha$ is \emph{outer}.
    Then, up to the identification of $A$ with its canonical image $\alpha(A) \subseteq B$, one can show that \[A \subseteq A \rtimes_{r, \alpha} \widehat{\G},\] is a unital, irreducible C*-discrete inclusion; see Appendix \ref{sec:appendix} for details.
    If $A$ is separable and has Property (SP), then by Theorem \ref{thm:discreteSP} so does $B =  A \rtimes_{r, \alpha} \widehat{\G}.$ 
\end{exmp}
    
We can also consider a dual class of examples.
In favour of a succinct exposition, we shall again gloss over technical details regarding compact quantum groups and their representations, which are well-known to experts.
We refer the interested reader to \cite{MR3204665, MR3675047, MukoharaThesis}.
\begin{exmp} 
     Let $\mathbb{G}$ be a compact quantum group and $B$ be a separable simple C*-algebra with a \emph{minimal} faithful action $\alpha:B\to B\otimes C(\mathbb{G}).$
     Recall the fixed-point algebra is given by
     $$
     A:=\{b\in B|\ \alpha(b) = b\otimes 1\},
     $$
     and we assume it is simple and stable.
     With the associated haar state $h$ on $\mathbb{G}$ we can endow the inclusion 
     $
     A\subseteq B
     $
    with the faithful expectation 
    $$
    E:= (1_B\otimes h)\circ \alpha: B \twoheadrightarrow A
    $$
    so that the resulting inclusion is irreducible and C*-discrete. 
    As such, we obtain the decomposition as $A$-$A$ correspondences
    $$
    \mathcal{E}\cong \bigoplus_{u \in \mathsf{Irr}(\mathbb{G})}B_u\otimes H_u.
    $$
    Here,$H_u$ are the so-called \emph{multiplicity spaces} which are finite dimensional Hilbert spaces of the same dimension as the irreducible representation $u,$ spanned by isometries in the multiplier algebra $M(B)$ satisfying certain commutation relations with the left $A$-multiplication. 
    The unitary tensor category of finite dimensional representations of $\mathbb{G}$ is denoted by $\mathsf{Rep}(\mathbb{G}),$ with a fixed complete set representatives of irreducible representations $\mathsf{Irr}(\mathbb G).$ 
    The $A$-$A$ correspondences $B_u$ are realized by endomorphisms $\rho_u: A\to A$ of finite index as ${}_{\rho_u}{A}_{A} \cong {}_{A}{B_u}_{A},$ where the right action is trivial and the left is implemented by $\rho_u.$

    In \cite{MR5051789} the $H_u$ are referred to as the the \emph{diamond spaces} parameterizing the bimodule embeddings $B_u\to \mathcal{E}$, and is shown that the assignment 
    \begin{align*}
        \mathbb{H}:\mathsf{Rep}(\mathbb{G})^{\rm{opp}}&\to \mathsf{Hilb}_{\mathsf{fd}}\\
        u&\mapsto H_u    
    \end{align*}
    gives a \emph{connected C*-algebra object} in the $\textsf{Ind}$-completion category $\mathsf{Vec}(\mathsf{Rep}(\mathbb{G}))$. 
    That is, the C*-tensor category $\mathsf{Vec}(\mathsf{Rep}(\mathbb{G}))$ consists of linear functors $\mathsf{Rep}(\mathbb{G})^{\rm{opp}}\to \mathsf{Vec}$ valued in complex vector spaces.
    For a detailed exposition on operator algebras in unitary tensor categories we refer the reader to \cite{MR3687214, MR3948170}.

    The assignment
    \begin{align*}
        F: \mathsf{Rep}(\mathbb G)&\to \mathsf{Bim}_{\mathsf{f}}(A)\\
        u&\mapsto B_u
    \end{align*}
    is a unitary tensor functor and thus a C*-algebra object in the C*-tensor category $\mathsf{Vec}(\mathsf{Rep}(\mathbb G)^{\rm{opp}})$ of linear functors $\mathsf{Rep}(\mathbb G)\to \mathsf{Vec}$ (after composing with the forgetful functor to vector spaces). 
    Together, $F$ and $\mathbb{H}$ yield the projective quasi-normalizer 
    $$
    PQN(A\subseteq B) = \bigoplus_{u\in \mathsf{Irr}(\mathbb G)}F(u)\otimes \mathbb{H}(u),
    $$
    witnessing C*-discreteness. 
    That is, 
    $
    PQN(A\subseteq B)\subseteq B
    $
    is dense in operator norm. 
    This then realizes $B$ as the reduced crossed product of the action $F$ of $\mathsf{Rep}(\mathbb G)$ on $A$ over the connected C*-algebra object $\mathbb{H}.$
    That is, 
    $$
    B\cong A\rtimes_{F,r}\mathbb{H}\subseteq \mathsf{End}(\mathcal{E}_{A}),
    $$
    where $\mathsf{End}(\mathcal{E}_{A})$ is the C*-algebra of adjointable endomorphisms of the right Hilbert $A$-module $\mathcal E.$

    Finally, if $A$ has Property (SP), then by Theorem \ref{thm:discreteSP}, so does $B$. 
\end{exmp}

\section{SP-inclusions from Boundaries}\label{sec:Boundaries}

We return now to the more general setting of SP-inclusions and consider an example which arises naturally from dynamics. 
In this section, $G$ always refers to a discrete group.

We let $\partial_F G$ denote the {\bf Furstenberg boundary} of $G$. The Furstenberg boundary is the universal compact $G$-space on which the action of $G$ is both \emph{minimal} and \emph{strongly proximal}. That is, it is the universal compact $G$-space such that the $G$-orbit $G_x := \{gx: g \in G\}$ is dense in $\partial_F G$ for all $x \in \partial_F G$ and for any pair of probability measures $\mu$ and $\nu$ on $\partial_F G$, there is a net $(g_i) \in G$ such that $\lim g_i\mu = \lim g_i\nu$. 
There is a natural embedding of the reduced group C*-algebra into the reduced crossed product of $C(\partial_F G)$ by $G$. 
We refer the reader to \cite{MR3652252} for further details regarding the Furstenberg boundary. 

Recall that a discrete group $G$ is called {\bf C*-simple} if $C_r^*(G)$ is simple.
We will freely use the various characterizations of C*-simplicity for discrete groups found in \cite{MR3652252}; in lieu of repeating precise statements of those results here, we provide references throughout. 

\begin{exmp}
    For each $n \in \mathbb{N}$, $n \geq 2$, the free group on $n$ generators, which we denote by $\mathbb{F}_n$, is C*-simple \cite{MR374334}. 
\end{exmp}

\begin{exmp}
     Recall that for $p$ prime, a \emph{Tarski monster group} of order $p$ is a finitely generated group in which every nontrivial subgroup is cyclic of order $p$. These exist for all primes $p > 10^{75}$ \cite{MR571100} and they are C*-simple \cite[Corollary 6.6]{MR3652252}. 
\end{exmp}

\begin{lem}\label{lem:csimpleSP}
    Let $G$ be a discrete C*-simple group and let $A := C(\partial_FG)$. Then $A \rtimes_r G$ has Property (SP). 
\end{lem}

\begin{proof}
   Recall that $G$ being C*-simple means that the boundary action of $G$ on $A := C(\partial_FG)$ is topologically free \cite[Theorem 1.5]{MR3652252}. Since $A$ is commutative, this is equivalent to the action being properly outer. In particular, by Lemma \ref{lem:properlyOuter},  $E: A \rtimes_r G \twoheadrightarrow A$ is an outer expectation.
   
    Moreover, it is well-known that $A = C(\partial_F G)$ is an AW$^*$-algebra (see e.g., \cite[Remark 3.16]{MR3652252}), hence has Property (SP). 
   As in Example \ref{exmp:properlyOuter}, it follows that $A \rtimes_r G$ has Property (SP).
\end{proof}

\begin{thm}\label{thm:csimpleSP}
    Let $G$ be a discrete C*-simple group. 
    Then 
    $$
    C_r^*(G) \subseteq C(\partial_F G) \rtimes_r G
    $$
    is an SP-inclusion.
\end{thm}

\begin{proof}
    This is immediate from Lemma \ref{lem:csimpleSP} and Theorem \ref{thm:mainSP}.
\end{proof}

This result raises the following question, which we were unable to resolve.

\begin{quest}
    Let $G$ be a (countable) discrete C*-simple group. Is the embedding $C_r^*(G) \subseteq C(\partial_F G) \rtimes_r G$ a RR0-inclusion?
\end{quest}

\smallskip
Finally, let us consider the inclusion $C_r^*(G) \subseteq C(\partial_F G) \rtimes_r G$ under the more general assumption that $G$ is a discrete, non-amenable group.  

\begin{exmp}\label{exmp:furstenberg}
    Consider the (countable) discrete, non-amenable group ${G := \F_2 \times \Z}$.
    By the standard theory of reduced group C*-algebras (see e.g., \cite[Example II.10.3.15]{Blackadar}), we know that 
    \[ C_r^*(G) \cong C_r^*(\F_2) \otimes C(\mathbb{T}) =: A, \] where $\otimes$ is the minimal (spatial) tensor product.
    
    We would also like to describe $C(\partial_F G) \rtimes_r G$ explicitly.    
    First, notice that the {\bf amenable radical} (that is, the largest normal amenable subgroup) of $\F_2 \times \Z$ is $\{e\} \times \Z$ because $\F_2$ is C*-simple (and therefore does not contain any nontrivial normal amenable subgroups). 
    It follows from \cite[Proposition 3.19]{MR3652252} that \[\partial_F(\F_2 \times \Z) \cong \partial_F\left((\F_2 \times \Z) /(\{e\} \times \Z)\right) \cong  \partial_F(\F_2).\] 
    
    It is also known that the amenable radical of a discrete group $G$ always acts trivially on its Furstenberg boundary \cite[Proposition 8]{MR1998109}. 
    In particular, the action of $\Z \cong \{e\} \times \Z$ is trivial on $C(\partial_F (\F_2 \times \Z)) = C(\partial_F\F_2).$ 
    
    We now consider the induced action of $\Z$ on $D:= C(\partial_F \F_2) \rtimes_r \F_2$ and show it is trivial.
    For each $n \in \Z$, the induced automorphism in $\operatorname{Aut}(D)$ is determined by 
    $$f \mapsto n \cdot f$$ for all  $f \in C(\partial_F G)$
    and $$(u_s) \mapsto u_{nsn^{-1}},$$ 
    for all standard unitaries $u_s, s \in \F_2$ implementing the action on $D$.
    Since $\Z$ acts trivially on $\partial_F \F_2$, we get $n \cdot f = f$ for all $f \in C(X)$ and since $G = \F_2 \times \Z$ is a direct product, $n \in \Z$ commutes with $s \in \F_2$, so $nsn^{-1} = s.$

    Putting everything together, we have 
    \begin{align*}
        C(\partial_FG) \rtimes_r G  = C(\partial_F (\F_2 \times \Z)) \rtimes_r (\F_2 \times \Z) &\cong (C(\partial_F \F_2) \rtimes_r\F_2) \rtimes_r \Z\\
        &\cong D \otimes C(\mathbb{T}) =: B.
    \end{align*}
Here, the first isomorphism involving iterated crossed products is an application of standard results (see e.g., \cite[Proposition 3.11]{MR2288954}) and the second isomorphism follows from the fact that the induced action of $\Z$ on $D = C(\partial_F \F_2) \rtimes_r \F_2$ is trivial (\cite[Example II.10.3.15(i)]{Blackadar}).

    Suppose that for any $0 \neq a \in A_+$, the hereditary subalgebra $\overline{a B a}$ contains a nonzero projection. 
    Take some nonzero $f \in C(\mathbb{T})_+$ with open support $U := \{z \in \mathbb{T}: f(z) \neq 0\}$, which is a nonempty proper open connected arc. 
    Use this function to define an element $0 \neq a:= 1 \otimes f \in A_+.$
    In particular, this means that
    \begin{align*}
        \overline{aBa} = \overline{(1 \otimes f)(D \otimes C(\mathbb{T}))(1 \otimes f)} &= D \otimes \overline{fC(\mathbb{T})f} \\
        &= D \otimes C_0(\{z \in \mathbb{T}: f(z) \neq 0\})  \\
        &= D \otimes C_0(U).
    \end{align*}
    
    Suppose that $D \otimes C_0(U)$ contains a nonzero projection $p$.
    Under the usual identification $D \otimes C_0(U) \cong C_0(U, D)$, we can view $p$ as a continuous $D$-valued function on $U$. 
    We will arrive at a contradiction by arguing as in Example \ref{exmp:Goodearl}.
    
    Since $p(z)$ is a projection for each $z \in U$, we have $\norm{p(z)} \in \{0,1\}$.
    By connectedness of $U$, the function $z \mapsto \norm{p(z)}$ is constant;
    since $p \neq 0$, we have $\norm{p(z)} = 1$ for all $z \in U$.
    But since $U$ is a proper open arc (hence non-compact), we must have that $p$ vanishes at infinity.
    But this contradicts the fact that $z \mapsto \norm{p(z)}$ is continuous and  $\norm{p(z)} = 1$ for all $z \in U$. 

    In other words, for our choice of $0 \neq a \in A_+$, the hereditary subalgebra $\overline{aBa} \subseteq B$ cannot contain a nonzero projection.
    In particular, \[ C_r^*(G) \cong C_r^*(\F_2) \otimes C(\mathbb{T}) \subseteq  (C(\partial_F \F_2) \rtimes_r\F_2) \otimes C(\mathbb{T}) \cong C(\partial_F G) \rtimes_r G\] is not an SP-inclusion. 
\end{exmp}

\begin{rem}\label{rem:InParticular}
    We note that the choice of $G$ in the preceding example was motivated by our desire to resolve a question raised in \cite[Remark 5.10]{GN25}.
    In particular, those authors asked whether $C_r^*(G) \subseteq C(\partial_F G) \rtimes_r G$ is a RR0-inclusion whenever $G$ is a countable, discrete, non-amenable group; Example \ref{exmp:furstenberg} shows that it need not even be an SP-inclusion, thereby providing a negative solution to the problem.
    
    We shall see that the same argument will go through for any discrete, non-amenable group $G:= N \times H$ where $N$ is nontrivial with trivial amenable radical and $H$ is a nontrivial abelian group which is \textbf{torsion-free} (that is, the only finite order element of $H$ is the identity).
    Whereas the proof is essentially the same as the details worked out in Example \ref{exmp:furstenberg}, we only give a sketch of the argument below.
\end{rem}

\begin{thm}[Theorem \ref{mainthm:Furstenberg}]\label{thm:Furstenberg}
    Let $N$ be a nontrivial discrete group with trivial amenable radical and let $H$ be a nontrivial discrete torsion-free abelian group. If $G:= N \times H$ then the embedding $C_r^*(G) \subseteq C(\partial_FG) \rtimes_r G$ is not an SP-inclusion. In particular, if $G$ is a countable, discrete, non-amenable group, then the embedding $C_r^*(G) \subseteq C(\partial_F G) \rtimes_r G$ need not be an SP-inclusion (and therefore not a RR0-inclusion). 
\end{thm}

\begin{proof}[Sketch of proof.]
    If $N \neq \{e_N\}$ is a discrete group with trivial amenable radical and $H \neq \{e_H\}$ is a discrete torsion-free abelian group, then it is clear that $G := N \times H$ is discrete and non-amenable. 
    Moreover, \[ C_r^*(G) = C_r^*(N \times H) \cong C_r^*(N) \otimes C(\widehat{H}),\] where $\widehat{H}$ is the \emph{Pontryagin dual} of $H$ (see e.g., \cite[Chapter 24]{MR551496} for details).
    Since $H \neq \{e_H\}$ is a discrete torsion-free abelian group, we know that $\widehat{H} = K \neq \{e_K\}$ is a compact connected abelian group \cite[Theorem 24.25]{MR551496}. 
    So the left-hand side of our inclusion is \[A := C_r^*(N) \otimes C(K).\]

    Now let us consider the other side of the inclusion. 
    We get
    \begin{align*}
        C(\partial_F G) \rtimes_r G = C(\partial_F(N \times H)) \rtimes_r (N \times H) &\cong (C(\partial_F N) \rtimes_r N) \rtimes_r H \\
        &\cong (C(\partial_F N) \rtimes_r N) \otimes C(\widehat{H}),
    \end{align*} where we have used the fact that $\{e\} \times H$ is the amenable radical of $G$ (since $N$ has trivial amenable radical and $H$ is amenable) and the standard isomorphism results cited in Example \ref{exmp:furstenberg}.
    The right-hand side of our inclusion is then \[B:= (C(\partial_F N) \rtimes_r N) \otimes C(K).\]

    Take a continuous function $0 \neq f \in C(K)_+$ with nonempty, proper open support $U:=\{z \in K: f(z) \neq 0\}$ and define $0 \neq a:= 1 \otimes f \in A_+$. 
    As above, the connectedness of $K \neq \{e_K\}$ acts as an obstruction to the existence of a nonzero projection in $\overline{aBa} \subseteq B.$
    So \[ A = C_r^*(N) \otimes C(K) \subseteq (C(\partial_F N) \rtimes_r N) \otimes C(K) = B\] is not an SP-inclusion.

    The ``in particular'' statement in Remark \ref{rem:InParticular} follows by taking $N$ and $H$ to be countable (as in Example \ref{exmp:furstenberg}) and by recalling that RR0-inclusions are necessarily SP-inclusions (Proposition \ref{prop:RR0}).
\end{proof}

We have seen that for a discrete group $G$, $C_r^*(G) \subseteq C(\partial_F G) \rtimes_r G$ is an SP-inclusion if $G$ is C*-simple (Theorem \ref{thm:csimpleSP}) but it need not be an SP-inclusion if we only assume that $G$ is non-amenable (Theorem \ref{thm:Furstenberg}).
Recall that the discrete groups with trivial amenable radical are a (strictly) intermediate class of groups between those which are C*-simple and those which are non-amenable \cite{MR3660307}. Equivalently, these are the discrete groups with {\bf unique trace property}, meaning that their reduced group C*-algebra admits a unique tracial state \cite[Theorem 1.3]{MR3735864}. 

\smallskip
We leave the following natural question(s) for future study.

\begin{quest}
    Let $G$ be a (countable) discrete group with unique trace property. Is $C_r^*(G) \subseteq C(\partial_FG) \rtimes_r G$ an SP-inclusion? Is it a RR0-inclusion?
\end{quest}

\appendix
\section{C*-discrete Inclusions from Discrete Quantum Groups}\label{sec:appendix}

In this appendix, we will justify that the inclusion described in Example \ref{exmp:discreteQuantum} is irreducible and C*-discrete.
We begin by recalling the construction in question.

\begin{exmp}
    Let $\G = (C(\G), \Delta)$ be a (reduced) \emph{compact quantum group}, let $\widehat{\G} = (c_0(\widehat{\G}), \widehat{\Delta})$ be the dual \emph{discrete quantum group}. In particular, we have that
    \begin{equation*}
        c_0(\widehat{\G}) = \bigoplus^{c_0}_{u \in \mathsf{Irr}(\G)} B(H_u)
    \end{equation*} where $\mathsf{Irr}(\G)$ refers to a fixed set of representatives of the equivalence classes of irreducible unitary representations of $\G$ and $B(H_u) \cong M_{n_u}(\mathbb{C})$ for $n_u := \dim H_u\in \mathbb{N}.$ 
    Let $A$ be a unital simple C*-algebra, let $\alpha: A \to M(c_0(\widehat{\G}) \otimes_{\min} A)$ be a \emph{left action} of $\widehat{\G}$ on $A$, that is
    an injective unital $\ast$-homomorphism satisfying 
    \begin{enumerate}[(i)]
        \item $(\widehat{\Delta} \otimes \operatorname{id})\alpha = (\operatorname{id} \otimes \alpha)\alpha$, and
        \item  $\overline{\operatorname{span}}^{\norm{\cdot}}\left\{(x \otimes 1)\alpha(a): x \in c_0(\widehat{\G}), a \in A\right\} = c_0(\widehat{\G}) \otimes_{\min} A$,
    \end{enumerate}
    and let $B :=  A \rtimes_{r, \alpha} \widehat{\G} = C^*(\alpha(A), C(\G) \otimes 1) \subseteq M(K(\ell^2(\widehat{\G})) \otimes A)$ be the reduced crossed product induced by the action. 
    
    Suppose that $\alpha$ is \emph{outer} (in the sense defined below).
    Then, up to the identification of $A$ with its canonical image $\alpha(A) \subseteq B$, one can show that \[A \overset{E}\subseteq A \rtimes_{r, \alpha} \widehat{\G}\] is a unital, irreducible C*-discrete inclusion.
    Here, $E: B \twoheadrightarrow A$ is the canonical faithful conditional expectation coming from the reduced crossed product construction. 
\end{exmp}

In order to show that this is an irreducible C*-discrete inclusion, we shall outline how to move between the operator algebraic setting of reduced crossed products by discrete quantum groups to the tensor category setting. 
We then appeal to the general theory of C*-discrete inclusions, which was developed in the language of the latter.

    Let $\mathsf{Rep}(\G)$ denote the UTC of finite-dimensional unitary representations of $\G$; every $U \in \mathsf{Rep}(\G)$ is (unitarily equivalent to) a finite direct sum of elements $u \in \mathsf{Irr(\G)}$ and we denote the trivial representation by $\mathbbm{1}$. 
    We also let $\mathsf{Corr}(A)$ denote the C*-tensor category of $A$-$A$ C*-correspondences, which contains the UTC $\mathsf{Bim_{f}}(A)$ as a full subcategory.
    Following the approach of \cite[Example 2.24]{MR4776902}, we will show that the action $\alpha: \widehat{\G} \curvearrowright A$ contains the same information as an action $\mathsf{Rep}(\G) \curvearrowright A$ via the construction of a \emph{unitary tensor functor} $F_\alpha: \mathsf{Rep}(\G) \to \mathsf{Bim_{f}}(A)$.
    We say that $\alpha$ is \textbf{outer} if $F_\alpha$ is fully faithful, following the convention of \cite[Definition 3.1]{MR5051789}.
    
    We make use of the categorical equivalence between $\mathsf{Rep}(\G)$ and the finite-dimensional nondegenerate $\ast$-representations $\pi_U: c_0(\widehat{\G}) \to B(H_U).$
    For each $U \in \mathsf{Rep}(\G)$ we may define a map $A \to B(H_U) \otimes A$ by setting \[\alpha_U := (\pi_U \otimes \operatorname{id}_A)\alpha. \] 
    Then $H_U \otimes A$ becomes an $A$-$A$ C*-correspondence when equipped with with left/right actions $$a \rhd (\xi \otimes b) = \alpha_U(a)(\xi \otimes b) \quad \mathrm{ and } \quad (\xi \otimes a) \lhd b = \xi \otimes (a b),$$ along with the $A$-valued inner product $$\braket{\xi \otimes a}{\eta \otimes b}_A = \braket{\xi}{\eta}_{H_U} a^*b.$$ 
    Moreover, since any \emph{intertwiner} $T : U \to V$ satisfies $T\pi_U(x) = \pi_V(x)T$,
 we get \[(T \otimes 1_A)\alpha_U(a) = \alpha_V(a)(T \otimes 1_A).\] 
 Therefore \[F_{\alpha}(U) = X_U := {}_{\alpha_U}{(H_U \otimes A)}_A\] and \[F_\alpha(T) = T \otimes 1_A\] defines a $\ast$-functor $F_\alpha: \mathsf{Rep}(\G) \to \mathsf{Corr}(A)$. 
    We may also define a unitary isomorphism $F_{\alpha}^1: A \to X_{\mathbbm{1}}$ by \[a \mapsto 1_\mathbb{C} \otimes a\] and a natural unitary isomorphism $F^2_{\alpha,U,V}: X_{U} \boxtimes X_V \to X_{U \otimes V}$ by \[((\xi \otimes a) \boxtimes(\eta \otimes b)) \mapsto \xi \otimes \alpha_V(a)(\eta \otimes b),\] under the convention that $$\pi_{U \otimes V} = (\pi_U \otimes \pi_V)\widehat{\Delta}.$$
    Hence, the triple $(F_\alpha, F_\alpha^1, F_\alpha^2)$ forms a unitary tensor functor; we will simply refer to $F_\alpha$ as the unitary tensor functor, with the other two maps being implicit.
    
    Since $(X_U)_A \cong A_A^{\oplus n_U},$ we see that each $X_U$ is a finitely generated projective right $A$-module.  
    Moreover, because $\mathsf{Rep}(\G)$ is rigid and $F_\alpha$ is a unitary tensor functor, every $F_\alpha(U) = X_U$ is dualizable; it follows that $X_U \in \mathsf{Bim_{f}}(A)$ for each $U \in \mathsf{Rep}(\G)$.
    In other words, we have constructed a unitary tensor functor \[F_\alpha: \mathsf{Rep}(\G) \to \mathsf{Bim_{f}}(A)\] which encodes the same information as the discrete quantum group action ${\alpha: \widehat{\G} \curvearrowright A}$.
    
    We also know, by Tannaka-Krein duality, that $\widehat{\G}$ is encoded by the \emph{connected quantum group C*-algebra object} $\mathbf{G} \in \mathsf{Vec}(\mathsf{Rep}(\G))$ (see e.g., \cite[Example 16]{MR3687214}).
    In particular, following \cite[Sub-Example 2]{MR3687214}, the forgetful fiber functor $\mathsf{Rep}(\G)\to \mathsf{Hilb}_{\mathsf{f}},$ equips $\mathsf{Hilb_f}$ with the structure of a left $\mathsf{Rep(\G)}$-module category via \[U \rhd K = H_U \otimes K. \]
    Then $\mathbf{G}$ is the connected C*-algebra object corresponding to $\mathbb{C} \in \mathsf{Hilb_f}$ with \[ \mathbf{G}(U) = \mathsf{Hilb_f}(U \rhd \mathbb{C}, \mathbb{C}) = \mathsf{Hilb_f}(H_U \otimes\mathbb{C}, \mathbb{C}) = \mathsf{Hilb_f}(H_U, \mathbb{C}) = H_U^*,\] as established in \cite[p.1184]{MR3687214}.
    The multiplication and involution in the algebraic realization of $\mathbf{G}$ recover those of the Hopf $\ast$-algebra 
    \[\mathcal{O}(\G) := \bigoplus^{\text{alg}}_{u \in \mathsf{Irr}(\G)} \mathcal{C}(u) \] 
    where 
    \begin{align*} \mathcal{C}(u) := \operatorname{span}\left\{c^u_{\varphi, \xi} := (\varphi \otimes \operatorname{id}_{C(\G)})(u(\xi \otimes 1)): \varphi \in H_u^*,\, \xi \in H_u\right\}.
    \end{align*}
    This Hopf $\ast$-algebra is a dense $\ast$-subalgebra of $C(\G)$.
    Moreover, since $\G$ is reduced,  $C(\G)$ is uniquely determined in the GNS completion of $\mathcal{O}(\G)$ with respect to the Haar state.

    We now define the algebraic discrete quantum group crossed product 
    \begin{align*}
        B_{\text{alg}}:=A \rtimes_{\text{alg}, \alpha} \widehat{\G} &= \operatorname{span}\{(x \otimes 1)\alpha(a): x \in \mathcal{O}(\G), a \in A\} \\
        &= \bigoplus_{u \in \mathsf{Irr}(\G)}^{\text{alg}} (\mathcal{C}(u) \otimes 1)\alpha(A)
    \end{align*} and compare it with the algebraic generalized crossed product of $A$ and $\mathbf{G}$ (via the unitary tensor functor $F_\alpha$) described in \cite[Construction 4.4]{MR5051789},
    \begin{align*}
        \mathbf{B}_{\text{alg}} := A \rtimes_{\text{alg}, F_\alpha} \mathbf{G} = \bigoplus_{u \in \mathsf{Irr}(\G)}^{\text{alg}} F_\alpha(u) \otimes \mathbf{G}(u) 
        &= \bigoplus_{u \in \mathsf{Irr}(\G)}^{\text{alg}}X_u\otimes H_u^* \\
        &= \bigoplus_{u \in \mathsf{Irr}(\G)}^{\text{alg}}{}_{\alpha_u}{(H_u \otimes A)}_A\otimes H_u^*.
    \end{align*} 
    That $B_{\text{alg}}$ and $\mathbf{B}_{\text{alg}}$ are isomorphic as vector spaces is essentially immediate from the definitions and the \emph{Peter-Weyl decomposition} of $\mathcal{O}(\G)$, which says that $\mathcal{C}(u) \cong H_u^* \otimes H_u$ for each $u \in \mathsf{Irr}(\G).$
    One can then verify that the map $\Theta:   \mathbf{B}_\mathrm{alg} \to B_\mathrm{alg} $ defined by \[ (\xi \otimes a)\otimes \varphi \mapsto (c^u_{\varphi, \xi} \otimes 1)\alpha(a) \] preserves the multiplication and involution described in \cite[Construction 4.4]{MR5051789}. 
    The restriction of $E$ to $B_{\text{alg}}$ is the algebraic conditional expectation $E_0: B_{\text{alg}} \twoheadrightarrow A$  given by $$E_0((c^u_{\varphi, \xi} \otimes 1)\alpha(a)) = h(c^u_{\varphi, \xi})a$$ where $c^u_{\varphi, \xi} \in \mathcal{C}(u)$ for some $u \in \mathsf{Irr}(\G)$, $a \in A$, and where $h$ is the Haar state; it is a projection onto the $u = \mathbbm{1}$ component. 
    On the categorical side, the conditional expectation $E_0': \mathbf{B}_{\text{alg}} \twoheadrightarrow A$ is defined to be the projection onto the $u = \mathbbm{1}$ component, so we get $E_0' = E_0\circ \Theta.$

    Following \cite[Construction 4.4]{MR5051789}, we can build a unital C*-algebra ${A \rtimes_{r, F_\alpha} \mathbf{G}}$ from our unital, simple C*-algebra $A$, unitary tensor functor ${F_\alpha: \mathsf{Rep}(\G) \to \mathsf{Bim_f}(A)}$, and C*-algebra object $\mathbf{G} \in \mathsf{Vec}(\mathsf{Rep}(\G))$.  
    The key idea for us is that this reduced crossed product contains the same information as $A \rtimes_{r, \alpha} \widehat{\G}.$ 
    In particular, the algebraic $\ast$-isomorphism above gives rise to a C*-algebraic $\ast$-isomorphism as follows.
    The generalized GNS completions $$\mathcal{E}' := L^2(\mathbf{B}_\mathrm{alg}, E'_0) \quad \text{ and } \quad \mathcal{E} := L^2(B_\mathrm{alg}, E_0)$$ can be identified by defining a map $W_0: \mathcal{E}' \to \mathcal{E}$ by $$z \Omega' \mapsto \Theta(z) \Omega$$ and extending to a unitary $W: \mathcal{E}' \to \mathcal{E}$. This unitary intertwines the respective left regular representations, so $\Theta$ extends to an expectation-preserving $\ast$-isomorphism at the C*-algebra level, as required.

    Finally, if $F_\alpha$ is fully faithful then \cite[Theorem 4.9]{MR5051789} states that \[A \overset{E'}\subseteq A \rtimes_{r, F_{\alpha}} \mathbf{G}\] is an irreducible C*-discrete inclusion, where $E': A\rtimes_{{r, F_\alpha}} \mathbf{G} \twoheadrightarrow A$ is a faithful conditional expectation which extends $E'_0$ \cite[Proposition 4.7]{MR5051789}.
    By above, this is equivalent to our claim that if the action $\alpha: \widehat{\G} \curvearrowright A$ is outer, then \[ A \overset{E}\subseteq A \rtimes_{r, \alpha} \widehat{\G}\] is an irreducible C*-discrete inclusion.   

\bibliographystyle{amsalpha}
\bibliography{bibliography}
\end{document}